\documentclass[10pt, a4paper]{amsart}

\usepackage[english]{babel}
\usepackage[utf8]{inputenc}
\usepackage{amsmath}
\usepackage{graphicx}
\usepackage[colorinlistoftodos]{todonotes}
\usepackage{amssymb,amscd,latexsym,epsfig,xy}
\usepackage[mathscr]{euscript}
\usepackage{amsthm}
\usepackage{tikz}
\usepackage{MnSymbol}
\usetikzlibrary{matrix,arrows,decorations.pathmorphing}
\usepackage{hyperref}
\usepackage{enumerate}
\usepackage{float}
\usepackage{cleveref}
\usepackage{cite}

\crefformat{section}{\S#2#1#3} 
\crefformat{subsection}{\S#2#1#3}
\crefformat{subsubsection}{\S#2#1#3}

\newtheorem{defn}{Definition}[subsection]
\newtheorem{thm}[defn]{Theorem}
\newtheorem{lem}[defn]{Lemma}

\newcommand\A{\mathbb A}
\newcommand\C{\mathbb C}

\newcommand\NN{\mathbb N}
\newcommand\PP{\mathbb P}
\newcommand\R{\mathbb R}
\newcommand\Q{\mathbb Q}
\newcommand\Z{\mathbb Z}
\newcommand\fd{\mathfrak d}
\newcommand\fD{\mathfrak D}
\newcommand\fm{\mathfrak m}
\newcommand\cA{\mathcal A}

\newcommand\cO{\mathcal O}
\newcommand\cP{\mathcal P}

\newcommand\cX{\mathcal X}

\newcommand\fg{\mathfrak g}

\newcommand\Ext{\operatorname{Ext}}

\newcommand\Hom{\operatorname{Hom}}

\newcommand\Rea{\operatorname{Re}\,}

\newcommand\Ima{\operatorname{Im}\,}

\newcommand\ch{\operatorname{ch}}
\newcommand\Pic{\operatorname{Pic}}

\newcommand\Stab{\operatorname{Stab}}
\newcommand\D{\operatorname{D}}
\newcommand\Arg{\operatorname{Arg}}
\newcommand\vir{\mathrm{vir}}

\title[Scattering diagrams, sheaves, and curves]{Scattering diagrams, sheaves, and curves}

\author{Pierrick Bousseau}

\date{}

\begin{document}

\begin{abstract}

We review the recent proof of the N.~Takahashi's conjecture on genus $0$
Gromov-Witten invariants of $(\PP^2,E)$,
where $E$ is a smooth cubic curve in the complex projective plane $\PP^2$.
The main idea is the use of the algebraic notion of scattering diagram as a bridge between the world of Gromov-Witten invariants of $(\PP^2,E)$ and the world of  moduli spaces of coherent sheaves on $\PP^2$. Using this bridge, the N.~Takahashi's conjecture can be translated into a manageable question about moduli spaces of coherent sheaves on $\PP^2$.

This survey is based on a three hours lecture series given as part of the 
Beijing-Zurich moduli workshop in Beijing, 9-12 September 2019.
\end{abstract}

\maketitle

\setcounter{tocdepth}{2}

\tableofcontents

\section{Introduction}
The main theme that we explore in the present review paper is the relationship
established in 
\cite{bousseau2019scattering, bousseau2019takahashi}
between two a priori
distinct geometric topics:
\begin{enumerate}
    \item Relative Gromov--Witten theory of the pair $(\PP^2,E)$, where $E$ is a smooth cubic curve in the complex projective plane $\PP^2$.
    \item Sheaf counting on $\PP^2$.
\end{enumerate}
The connecting link is provided by the algebraic notion of the scattering diagram.
Once the relationship established, it becomes possible to transfer information from one side to the other and to prove non-trivial results. We will only survey some of the results contained in \cite{bousseau2019scattering, bousseau2019takahashi}. 
In particular, we do not discuss higher-genus Gromov--Witten invariants and refined sheaf counting, for which we refer to \cite{bousseau2019scattering, bousseau2019takahashi, bousseau2019hae}.

Our correspondence through scattering diagrams between relative Gromov-Witten theory of $(\PP^2,E)$
and moduli spaces of coherent sheaves is inspired by and similar to  
the correspondence through scattering diagrams between log Gromov-Witten theory of log Calabi-Yau surfaces with maximal boundary and moduli spaces of quiver representations, which is nicely reviewed by Gross and Pandharipande in \cite{MR2662867}, following the work by Gross-Pandharipande-Siebert \cite{MR2667135}.

\subsection{Gromov--Witten theory of $(\PP^2,E)$}
\label{section_gw}
Let $E$ be a smooth cubic curve in the complex projective plane $\PP^2$. For every positive 
integer $d$, a general degree $d$ curve in $\PP^2$ intersects $E$ in $3d$ distinct points. 
Therefore, we expect that asking for degree $d$ curves intersecting $E$ at a single point defines a constraint of codimension $3d-1$ in the space of degree $d$ curves.
On the other hand, the space of rational degree $d$ curves in 
$\PP^2$ is of dimension $3d-1$. Thus, the space of rational degree $d$ curves 
in $\PP^2$ intersecting $E$ at a single point has expected dimension zero,
and the count of such curves should be a well-posed enumerative question.

In fact, this naive dimension counting gives the correct answer: there are really only finitely many rational degree $d$ curves intersecting $E$
at a single point. 

\begin{lem} \label{lem_isolated_curves}
There does not exist positive dimensional families of rational curves in $\PP^2$ meeting 
$E$ at a single point.
\end{lem}

\begin{proof}
If such a family existed, then one could construct a curve $B$ and a dominant rational map 
$f \colon \PP^1 \times B \dashedrightarrow \PP^2$ such that $f^{-1}(E) \subset \{ \infty \} \times B$. As $E$ is anticanonical in $\PP^2$, there exists a 2-form $\omega$, non-degenerate on $\PP^2-E$ and with first order pole along $E$. 
As we are working in characteristic zero, the pullback 
$f^{*} \omega$ is a non-degenerate  2-form on 
$(\PP^1 -\{\infty \}) \times B$, with first order pole along 
$\infty \times B$. As $B$ is curve, there exists a non-vanishing vector field
on some non-empty open subset $U$ of $B$. Contracting the pullback of this vector field
to $\PP^1 \times U$ with $f^{*} \omega$, we get a non-zero 1-form on each $\PP^1$ fiber above $U$,
with only a first order pole at $\infty$ as singularity. As $\PP^1$ does not admit
non-zero 1-forms with only a first order pole at $\infty$ as singularity, this is a contradiction.
\end{proof}

One can view the pair $(\PP^2,E)$ as a log K3 surface. Lemma \ref{lem_isolated_curves}
is the analogue for $(\PP^2,E)$ of the fact that a K3 surface is not uniruled (in characteristic zero). The proofs are essentially the same in both cases, using the existence of a non-degenerate 2-form on a K3 surface or of a non-degenerate log 2-form on 
$(\PP^2,E)$. Counting rational curves in $\PP^2$ intersecting $E$ at a single point 
is a log version of counting rational curves in K3 surfaces.

Once we know that there are finitely many rational degree $d$ curves in $\PP^2$ intersecting $E$ at a single point, one can count them. However, this naive 
enumerative count has a major defect: it is not deformation invariant. In other words, it depends on the chosen cubic $E$. Rational curves intersecting $E$ at a single point are in general very singular, and one should count them with appropriate multiplicities in order to get a deformation invariant result.

Gromov--Witten theory provides a systematic way to set up deformation invariant enumerative questions. For every positive integer $d$, we have a moduli space 
$\overline{M}_0(\PP^2/E,d)$ of relative stable maps, which is a compactification of the space of degree $d$
maps $f \colon \PP^1 \rightarrow \PP^2$ such that 
$f^{-1}(E)=\{ \infty\}$. The moduli space $\overline{M}_0(\PP^2/E,d)$
is a proper Deligne-Mumford stack and
comes with a zero-dimensional virtual fundamental class
$[\overline{M}_0(\PP^2/E,d)]^{\vir}$. The corresponding Gromov--Witten invariant $N_{0,d}$ is the degree of this class, written as
\[ N_{0,d}^{\PP^2/E} \coloneq \int_{[\overline{M}_0(\PP^2/E,d)]^{\vir}} 1 \,.\]
In general, the virtual fundamental class is a zero-cycle with rational coefficients and so the Gromov--Witten invariant $N_{0,d}^{\PP^2/E}$
is a rational number. By deformation invariance of the virtual fundamental class, 
the relative Gromov--Witten invariants $N_{0,d}^{\PP^2/E}$ are deformation invariant: they do not depend
on the specific choice of the smooth cubic $E$.

The moduli space $\overline{M}_0(\PP^2/E,d)$ is not zero-dimensional in general. Indeed, a relative stable map $f: C \rightarrow (\PP^2,E)$ is in general very far from being an immersion. There are two major issues:
\begin{enumerate}
    \item Multiple cover contributions. Even if $f \colon  \PP^1 \rightarrow \PP^2$ is a nicely immersed degree $d$ rational curve in $\PP^2$, it will contribute to Gromov--Witten theory in every degree $kd$ multiple of $d$ through maps of the form $f \circ h$ where $h \colon \PP^1 \rightarrow \PP^1$ is a degree $k$ map.
    \item Contracted components. There are in fact two possible technical definitions of $\overline{M}_0(\PP^2/E,d)$:
    either using relative stable maps of J. Li \cite{MR1882667}, or using stable log maps of Abramovich-Chen-Gross-Siebert
    \cite{MR3224717, MR3257836, MR3011419}. In relative stable map theory, an element 
    $f \colon C \rightarrow (\PP^2,E)$ of 
    $\overline{M}_0(\PP^2/E,d)$ is a map 
    $f \colon C \rightarrow \PP^2 [n]$ for some $n \in \NN$,
    where $\PP^2 [n]$ is an expansion of $\PP^2$ obtained by $n$-successive degenerations to the normal cone of $E$. 
    In stable log map theory, an element $f \colon C \rightarrow (\PP^2,E)$ is an ordinary map
    $f \colon C \rightarrow \PP^2$
    but promoted at the level of log schemes. These two approaches produce moduli spaces $\overline{M}_0(\PP^2/E,d)$ which are in general slightly different, but define the same Gromov--Witten invariants $N_{0,d}^{\PP^2/E}$.
    
    Whatever the precise approach used, the moduli space $\overline{M}_0(\PP^2/E,d)$ contains in general maps 
    $f \colon C \rightarrow (\PP^2/E)$
    which are more complicated than maps $f \colon \PP^1 \rightarrow \PP^2$.
    Here is an example of what can happen. Let
    $f_1 \colon \PP^1 \rightarrow \PP^2$
    and $f_3 \colon \PP^1 \rightarrow \PP^2$ be immersed curves of degree $d_1$ and $d_3$, and intersecting $E$ at
    a single common point $p$. Let $C$ be a chain $C_1 \cup C_2 \cup C_3$ of three $\PP^1$s. Then, there are maps 
    $f \colon C \rightarrow (\PP^2,E)$
    of degree $d_1+d_3$, coinciding with $f_1$ on $C_1$, with $f_3$ on $C_3$ and mapping $C_2$ inside a ``bubble" in the relative stable map language, or contracting $C_2$
    onto $p$ in the log language.
\end{enumerate}

These two issues, multiple covers and contracted components, are the price to pay in Gromov--Witten theory for deformation invariance. As the Gromov--Witten invariants 
$N_{0,d}^{\PP^2/E}$ are defined through a virtual fundamental class construction on  possibly higher-dimensional stacky moduli spaces, their direct geometric meaning is quite unclear.

In order to understand more precisely when multiple covers and contracted components occur, we need to make a simple observation. Let 
$p_0$ be one of the 9 flex points of $E$ and let $L$ be the tangent line to $E$ at $p_0$. Let $C \subset \PP^2$ be a degree $d$ curve intersecting $E$ at a single point $p$. The curve $C$
is linearly equivalent to 
$dL$ in $\PP^2$. Intersecting this relation with $E$, we get that $3d p$ is linearly equivalent to $3d p_0$ in $E$, i.e.\ we have the relation 
$3d(p-p_0)=0$ in $\Pic^0(E)$. Therefore, the point of contact of $C$ with $E$
necessarily belongs to the set $P_d$ of the $(3d)^2$
points $p$ in $E$ such that $p-p_0$ is $3d$-torsion in $\Pic^0(E)$. The definition of $P_d$ is independent of the choice of the flex point $p_0$. Indeed, if $p_0'$ is another flex point, then $p_0'-p_0$ is $3$-torsion in $\Pic^0(E)$.

It follows that the image of the evaluation morphism $ \overline{M}_0(\PP^2/E,d) \rightarrow \PP^2$ at the contact point with $E$ is contained in $P_d$. It is true even in the presence of contacted components because we are working in genus-0 Gromov--Witten theory and a rational curve cannot dominate $E$.

Therefore, the moduli space $ \overline{M}_0(\PP^2/E,d)$ splits into disjoint components indexed by the points $p \in P_d$:
\[\overline{M}_0(\PP^2/E,d) = \coprod_{p \in P_d} \overline{M}_0(\PP^2/E,d)^p \,.\]
Restricting the virtual fundamental class $[\overline{M}_0(\PP^2/E,d)]^{\vir}$ to the various components, we get 
virtual fundamental classes 
$[\overline{M}_0(\PP^2/E,d)^p]^{\vir}$
and we define
\[ N_{0,d}^{\PP^2/E,p} \coloneq 
\int_{[\overline{M}_0(\PP^2/E,d)^p]^{\vir}} 1 \in \Q\,. \]
For every $p \in P_d$, $N_{0,d}^{\PP^2/E,p}$ is the contribution to $N_{0,d}^{\PP^2/E}$
of the rational degree $d$ curves meeting $E$ at $p$. We have 
\[ N_{0,d}^{\PP^2/E}=\sum_{p \in P_d}
N_{0,d}^{\PP^2/E,p}\,.\]

The splitting according to the point $p$ is useful to understand the geometry underlying the Gromov--Witten invariants because the presence of multiple covers or contracted components depends strongly on the point $p \in P_d$. The invariant
$N_{0,d}^{\PP^2/E,p}$
receives contributions from degree $d'$ dividing $d$ through multiple covers only if 
$p \in P_{d'}$. This motivates the following definition.

\begin{defn}
For every $p \in \bigcup_{d \geqslant 1} P_d$, we denote
by $d(p)$ the smallest $d$ such that $p \in P_d$.
\end{defn}

Let $p \in P_d$. In general, $d(p)$ is a divisor of $d$. If $d(p)=d$, then $p$ is said to be \emph{primitive}. In such case, there are no multiple covers and no contracted components, the moduli space $\overline{M}_0(\PP^2/E,d)^p$ is zero-dimensional and so consists in finitely many (possibly non-reduced)
points. In particular, $N_{0,d}^{\PP^2/E,p}$ is the number of these points (weighted by their length if non-reduced) and so 
$N_{0,d}^{\PP^2/E,p}$ is a nonnegative integer. Thus, if $p$ is primitive, $N_{0,d}^{\PP^2/E,p}$ is as close as possible to the naive enumeration of rational curves in $\PP^2$: each curve is counted with an integer multiplicity.

If $d(p) \neq d$, then $p$ is said to be \emph{non-primitive}. In such case, there are
in general multiple covers and contracted components, $N_{0,d}^{\PP^2/E}$ is only a rational number, and its direct geometric meaning is unclear. The worst case (if $d>1$)
is in some sense $d(p)=1$, i.e.\ 
if $p$ is one of the flex points.
Through multiple covers, the tangent line to a flex point $p$ contributes to the invariants $N_{0,d}^{\PP^2/E,p}$
in every degree $d \geqslant 1$.

\begin{lem} \label{lem_monodromy}
For every positive integer $d$, 
the invariants 
$N_{0,d}^{\PP^2/E,p}$ depends on $p$ only through $d(p)$.
\end{lem}

\begin{proof}
If $p$ and $p'$ are two points in $P_d$ with $d(p)=d(p')$, then the monodromy of the family of all smooth cubics in $\PP^2$ is big enough to map $p$ on $p'$ and the result follows by deformation invariance of the Gromov--Witten invariants.
\end{proof}

For every positive integer $d$ and for every $k$ positive integer dividing $d$, we write
$N_{0,d}^{\PP^2/E,k}$ for $N_{0,d}^{\PP^2/E,p}$, with
$p \in P_d$ such that $d(p)=d$. This makes sense by Lemma \ref{lem_monodromy}.
With this new notation, the primitive invariants are the $N_{0,d}^{\PP^2/E,d}$, and the non-primitive invariants are the 
$N_{0,d}^{\PP^2/E,k}$ with $k<d$.

The following result computes the non-primitive invariants in terms of the primitive ones.

\begin{thm}\cite{bousseau2019takahashi} \label{thm_gw}
For every positive integer $d$ and for every $k$ positive integer dividing $d$, we have 
\[ (-1)^{d-1}N_{0,d}^{\PP^2/E,k}
=\sum_{\substack{d' \\ k|d'|d}} 
\frac{1}{(d/d')^2} (-1)^{d'-1}
N_{0,d'}^{\PP^2/E,d'} \,.\]
\end{thm}

In order to get rid of signs, we define
\[ \overline{\Omega}_{d,k}^{\PP^2/E}
\coloneq (-1)^{d-1}N_{0,d}^{\PP^2/E,k}\,.\]
We define ``BPS invariants"
$\Omega_{d,k}^{\PP^2/E}$
by the formula
\[ \overline{\Omega}_{d,k}^{\PP^2/E}=
\sum_{\substack{d' \\ k|d'|d}}
\frac{1}{(d/d')^2} \Omega_{d',k}^{\PP^2/E}\,.\]
We can rephrase Theorem 
\ref{thm_gw} as follows.

\begin{thm}\cite{bousseau2019takahashi} \label{thm_gw-bps}
For every positive integer $d$, the BPS invariant
$\Omega_{d,k}^{\PP^2/E}$ is independent of $k$.
\end{thm}

Our main goal is to explain a proof of Theorem \ref{thm_gw-bps}.

The study of the Gromov--Witten counts 
$N_{0,d}^{\PP^2/E}$ was initiated by N.\ Takahashi \cite{takahashi9605007curves, MR1844627}
around 1999 and some form of 
Theorem \ref{thm_gw-bps}
was then conjectured. 
A more recent study of this question has been done by Choi--van Garrel--Katz--Takahashi \cite{choi2018log, choi2018local, choi2019contributions}.
In particular, the statement of Theorem \ref{thm_gw-bps} can be found as 
\cite[Conjecture 1.3]{choi2018log}. The natural analogue of Theorem \ref{thm_gw-bps} should hold for any pair $(S,D)$ with $S$ a del Pezzo surface and $D$ a smooth anticanonical divisor. In the present paper, we focus on $(\PP^2,E)$.

We have already explained that counting rational curves in $\PP^2$
intersecting $E$ at a single point should be viewed as an analogue to counting rational curves in K3 surfaces. Let $S$ be a projective K3 surface and let $\beta$ be an effective curve class on $S$. Then, one defines a Gromov--Witten count 
$N_{0,\beta}^S \in \Q$ of rational curves in $S$ of class $\beta$, which is invariant under deformations of $S$ keeping $\beta$ effective \cite{MR2746343}. 
Using the monodromy in the moduli space of K3 surfaces, one can show that $N_{0,\beta}$ only depends on 
$\beta^2$ and on the divisibility of 
$\beta$ in the lattice 
$H_2(S,\Z)$.
The divisibility of $\beta$ for K3 surfaces is analogous to the choice of the point $p \in P_d$ for $(\PP^2,E)$:
if the divisibility of $\beta$ is $1$, i.e.\ if $\beta$ is primitive, then 
$N_{0,\beta}^S$ is a positive integer, counting rational curves with integer multiplicities, whereas if $\beta$ is non-primitive, $N_{0,\beta}^S$ is only a rational number, receiving complicated contributions from multiple covers.

One defines ``BPS invariants" $n_{0,\beta}^S$
by the formula \[ N_{0,\beta}^S=\sum_{\beta=k\beta'}
\frac{1}{k^3}n_{0,\beta'}^S\,.\]

The following result is due to
Klemm-Maulik-Pandharipande-Scheidegger
\cite{MR2669707}
in 2010 and is the analogue of Theorem \ref{thm_gw-bps} for K3 surfaces.

\begin{thm}
For every $\beta$ effective curve class on $S$, 
the BPS invariant $n_{0,\beta}$ is independent of the divisibility of $\beta$, i.e.\ depends on $\beta$
only through $\beta^2$.
\end{thm}

\subsection{Dimension $1$ sheaves on $\PP^2$}
We introduce now a topic seemingly disjoint from the questions discussed in \cref{section_gw}.
We consider coherent sheaves $F$ on $\PP^2$ supported on curves of degree 
$d(F)$ and of Euler characteristic  
$\chi(F)$.
We recall that such coherent sheaf $F$ is said to be Gieseker semistable (resp.\ stable) if:
\begin{enumerate}
    \item $F$ is pure of dimension $1$, i.e.\ if every non-zero subsheaf of $E$ is also supported in dimension $1$.
    \item For every non-zero and strict subsheaf $F'$ of $F$, we have 
    \[ \frac{\chi(F')}{d(F')} \leqslant \frac{\chi(F)}{d(F)}\,\]
    (resp.\ $\frac{\chi(F')}{d(F')} < \frac{\chi(F)}{d(F)}$).
\end{enumerate}

For every positive integer $d$ and for every integer $\chi$, let $M_{d,\chi}$ be the moduli space of 
$S$-equivalence classes of Gieseker semistable sheaves $F$ on $\PP^2$, supported on curves of degree $d$ and with $\chi(F)=\chi$. Such moduli space can be constructed by geometric invariant theory. We refer to \cite{huybrechts2010geometry} for details. One can show \cite{MR1263210} that 
$M_{d,\chi}$ is an irreducible normal projective variety of dimension $d^2+1$. Taking the support defined by the Fitting ideal of a dimension $1$ sheaf defines a morphism
\[ \pi \colon M_{d,\chi} \rightarrow |\cO(d)|\,,\]
where $|\cO(d)|$ is the linear system of degree $d$ curves in $\PP^2$.
If $C \in |\cO(d)|$ is smooth, then 
$\pi^{-1}(C)$ is isomorphic to the Jacobian variety of $C$. If $C$ is singular, then $\pi^{-1}(C)$ is in general complicated, and this makes the study of the global geometry of 
$M_{d,\chi}$ non-trivial.

If $F$ is Gieseker semistable sheaf, then 
$\Ext^2(F,F)=\Hom(F,F \otimes K_{\PP^2})^\vee$ by Serre duality 
and $\Hom(F,F \otimes K_{\PP^2})^\vee=0$ by negativity of $K_{\PP^2}$ and semistability of 
$F$. In particular, the locus $M_{d,\chi}^{st} \subset M_{d,\chi}$ of stable objects is always smooth.
If $d$ and $\chi$ are coprime, then
$M_{d,\chi}=M_{d,\chi}^{st}$ and so 
$M_{d,\chi}$ is smooth. In general, there are strictly semistable sheaves and
$M_{d,\chi}$ is singular.

For every $d$ and $\chi$, we denote by $Ie(M_{d,\chi})$ the Euler characteristic of $M_{d,\chi}$ for the intersection cohomology and we define
\[ \Omega_{d,\chi}^{\PP^2}\coloneq (-1)^{d-1} Ie(M_{d,\chi}) \in \Z \,.\]
When $M_{d,\chi}$ is smooth, intersection cohomology coincides with singular cohomology and so $Ie(M_{d,\chi})$ coincides with the ordinary topological Euler characteristic.

Tensoring by $\cO(1)$ induces an isomorphism $M_{d,\chi} \simeq M_{d,\chi+d}$, and so $M_{d,\chi}$
only depends on $\chi$ modulo $d$. Similarly, Serre duality implies
that 
$M_{d,\chi} \simeq M_{d,\chi'}$
if $\chi=-\chi' \mod d$. However, if 
$d \geqslant 3$ and 
$\chi \neq \pm \chi' \mod d$, then the algebraic varieties 
$M_{d,\chi}$ and $M_{d,\chi'}$ are not isomorphic. Indeed, it is known by \cite{woolf2013nef} that they have different nef cones. This makes the following result quite remarkable:

\begin{thm}\cite{bousseau2019takahashi} \label{thm_sheaves} 
For every positive integer $d$, 
$\Omega_{d,\chi}^{\PP^2}$ is independent of $\chi$.
\end{thm} 

\begin{proof}
We give a sketch of proof. 
One first proves
\cite{bousseau2019takahashi} that the invariants 
$\Omega_{d,\chi}^{\PP^2}$
coincide with the dimension $1$ sheaves DT invariants of local 
$\PP^2$ defined by Joyce-Song 
\cite{MR2951762}. Thus, Theorem 
\ref{thm_sheaves} becomes a special case of a general conjecture of Joyce-Song, which 
by Toda 
\cite[Theorem 6.4]{MR2892766} is equivalent to the strong rationality conjecture for stable pair PT invariants
(see \cite[Conjecture 3.14 ]{MR2545686} and 
\cite[Conjecture 6.2]{MR2892766}). 

Given the known DT/PT correspondence (proved by wall-crossing in the derived category
for general Calabi-Yau 3-folds \cite{MR2813335}, or by computation of both sides in the toric case, see
\cite[\S 5]{MR2746343}), the strong rationality conjecture for PT invariants can be translated into a strong rationality statement for DT invariants.

As $K_{\PP^2}$ is a toric Calabi-Yau 3-fold, 
DT invariants can be computed by localization and organized using the topological vertex formalism \cite{MR2264664}.
By a study of the explicit formulas coming from the topological vertex formalism, Konishi  \cite[Theorem 1.3]{MR2215440}
proved that the strong rationality statement holds for local toric surfaces, 
and so in particular for $K_{\PP^2}$
(the proof was later generalized to arbitrary toric Calabi-Yau 3-folds in \cite{MR2250076}).
\end{proof}

\subsection{Main result}

Theorem \ref{thm_gw-bps} and Theorem
\ref{thm_sheaves} are formally quite similar, despite dealing with rather different geometric objects.
Theorem \ref{thm_gw-bps} is about understanding contributions of multiple covers and contracted components in Gromov--Witten theory of $(\PP^2,E)$, whereas Theorem \ref{thm_gw-bps} is about understanding contributions of strictly semistable sheaves in
dimension $1$ sheaf counting on $\PP^2$.

Our main goal is to give a survey of the proof of the following result.

\begin{thm}\label{thm_main} \cite{bousseau2019takahashi}
Theorem \ref{thm_sheaves} is equivalent to Theorem \ref{thm_gw-bps}.
\end{thm}

As we have already proved Theorem \ref{thm_sheaves}, this gives a proof of Theorem \ref{thm_gw-bps}.

The proof of Theorem \ref{thm_main} 
relies on the fact that the same algorithm, under the form of a scattering diagram, computes the invariants $\Omega_{d,k}^{\PP^2/E}$ and $\Omega_{d,\chi}^{\PP^2}$.
On the Gromov--Witten side, the scattering diagram will appear as tropical description of a normal crossing degeneration of $(\PP^2,E)$ \cite{gabele2019tropical}. On the sheaf side, the scattering diagram will appear as describing wall-crossing in the space of Bridgeland stability conditions on the derived category of coherent sheaves 
on $\PP^2$ \cite{bousseau2019scattering, bousseau2019takahashi}.

\subsection{Acknowledgment}
I thank Xiaobo Liu, Rahul Pandharipande, Emanuel
Scheidegger, and Qizheng Yin for the organization of the Beijing-Zurich moduli workshop.
I thank Michel van Garrel for sharing his notes of my lectures.

I acknowledge the support of Dr. Max Rössler, the Walter Haefner Foundation
and the ETH Zurich Foundation.

\section{Scattering diagrams and curve counting}

\subsection{Local scattering diagrams}
\label{section_local_scattering}
We follow \cite{MR2846484} and \cite{MR2667135}.
Let $M \simeq \Z^2$ be a two-dimensional 
lattice. Let $\fg =\bigoplus_{m \in M} \fg_m$
be a $M$-graded Lie algebra: we have a Lie bracket $[-,-]$ on $\fg$ such that 
$[\fg_m,\fg_{m'}]\subset \fg_{m+m'}$ for every 
$m, m' \in M$. We assume that 
$[\fg_m,\fg_{m'}]=0$ if $m$ and $m'$ are parallel.

Let $R$ be an Artinian local $\C$-algebra
with maximal ideal $\fm_R$. One can think 
about $R=\C[t]/t^N$ and $\fm_R=t R$. Then 
$\fg \otimes \fm_R$ is naturally a nilpotent Lie algebra for the bracket defined by 
$[g \otimes t, g' \otimes t]=[g,g'] \otimes tt'$. We denote by
$G \coloneq \exp(\fg \otimes \fm_R)$ the corresponding nilpotent group.
Concretely, elements of $G$ are elements of the form $e^g$, $g \in \fg$, and the product 
$e^g e^{g'}$ is defined by the Baker-Campbell-Hausdorff formula.

\begin{defn}
A \emph{ray} is a pair $(\fd, H_\fd)$, where:
\begin{enumerate}
    \item $\fd$ is an oriented half-line
    in $\R^2 \simeq M \otimes \R$, starting at $0$. We say that $\fd$ is ingoing if it points towards $0$ or outgoing it it points away from $0$.
    \item $H_\fd$ is an element of 
    $\fg \otimes \fm_R$ such that, writing 
    $H_{\fd}=\sum_j H_j$ with $H_j \in g_{m_j}$ and $H_j \neq 0$, all the elements $m_j \in M$
    are negatively collinear with the direction of $\fd$.
\end{enumerate}
\end{defn}

\begin{defn}
A \emph{local scattering diagram} $\fD$ is a finite collection of rays $(\fd, H_\fd)$.
\end{defn}

\begin{defn}
Let $\fD$ be a local scattering diagram.
We say that $\fD$ is \emph{consistent} if 
\[\vec{\prod_{(\fd,H_\fd)}} \exp (H_\fd)^{\epsilon_\fd} =1 \]
in $G$, where the ordered product is taken 
over the rays in the anticlockwise order, and where $\epsilon_\fd=+1$ if $\fd$ is outgoing 
and $\epsilon_\fd=-1$ if $\fd$ is ingoing.
\end{defn}

We adopt the normalization to identify two rays $(\fd, H)$ and $(\fd,H')$ with the same support $\fd$ to form a new ray 
$(\fd, H+H')$.

The following result goes back to Kontsevich-Soibelman
\cite{MR2181810}.

\begin{thm} \label{KS_lemma}
Let $\fD$ be a local scattering diagram. Then there exists a single consistent local scattering diagram $S(\fD)$ obtained from $\fD$ by adding only outgoing rays.
\end{thm}

\begin{proof}
We prove the result in $R/\fm_R^k$ by induction on $k$.

For $k=1$, we have $H_\fd=0 \mod \fm_R$, so $\exp(H_\fd)=1$ for every ray $(\fd,H_\fd)$, and so every local scattering diagram 
is consistent modulo $\fm_R$.

We assume that we have constructed 
$S(\fD) \mod \fm_R^k$, with rays 
$(\fd,H_\fd)$ such that 
\[\vec{\prod_{(\fd,H_\fd)}} \exp (H_\fd)^{\epsilon_\fd} =1 \mod \fm_R^k \,.\]
Then, we can uniquely write
\[\vec{\prod_{(\fd,H_\fd)}} \exp (H_\fd)^{\epsilon_\fd} =\exp (-\sum_j g_j) \]
for some $g_j \in g_{m_j} \otimes \fm_R^k$. We obtain $S(\fD) \mod \fm_R^{h+1}$ by adding the outgoing rays 
$(-\R_{\geqslant 0}m_j, g_j)$. This new local scattering diagram is consistent by construction, because 
$[\fg \otimes \fm_R^k, \fg \otimes \fm_R^k] \subset \fg \otimes \fm_R^{2k} \subset 
\fg \otimes \fm_R^{k+1}$ and so all the rays commute modulo $\fm_R^{k+1}$.
\end{proof}

\textbf{Examples}
\begin{enumerate}
\item Propagation of rays. Let 
$\fD$ be the local scattering diagram consisting of a single ingoing ray 
$(\fd=\R_{\geqslant 0}m, H_\fd)$. Then 
$S(\fD)$ is obtained by adding the outgoing ray 
$(-\R_{\geqslant 0}m, H_\fd)$, i.e.\ one propagates the ingoing ray.
\item Elementary scattering. We take 
$R=\C[t_1,t_2]/(t_1^2,t_2^2)$ and $\fD$ the local ingoing diagram consisting of two ingoing rays $(\R_{\geqslant 0}m_1, H_1)$
and $(\R_{\geqslant 0}m_2, H_2)$, propagating into two outgoing rays 
$(-\R_{\geqslant 0} m_1,H_1)$ and 
$(-\R_{\geqslant 0} m_2,H_2)$. We assume that 
$m_1$ and $m_2$ are primitive in $M$, and that $H_1 \in \fg_{m_1} \otimes t_1 R$ and $H_2 \in \fg_{m_2} \otimes t_2 R$.
In particular, we have $H_1^2=H_2^2=0$ and so 
\[ \exp(-H_2)\exp(-H_1)\exp(H_2) \exp(H_1)
=(1-H_2)(1-H_1)(1+H_2)(1+H_1)\]
\[=1-[H_1,H_2]=\exp(-[H_1,H_2])\,.\]
It follows that $S(\fD)$ is obtained by adding the outgoing ray 
\[(-\R_{\geqslant 0} (m_1+m_2), [H_1,H_2])\,.\]
\end{enumerate}

Computation of the consistent completion $S(\fD)$ of a general local scattering diagram $\fD$ can always be reduced to the computation of several elementary scatterings using a perturbation trick
\cite{MR2667135}. Assume that we work with 
$R =\C[t]/t^{N+1}$. We have a natural embedding 
\[ \C[t]/t^{N+1} \hookrightarrow 
\C[u_1,\dots,u_N]/(u_1^2,\dots, u_N^2)\]
\[ t \mapsto \sum_{j=1}^N u_j \,.\]
If $(\fd, H_\fd)$ is one of the rays of 
$\fD$, we can write, after the change of variables $t=\sum_{j=1}^N u_j$:
\[ H_\fd=\sum_{k} H_{\fd,k} \,,\]
where each $H_{\fd,k}$ is proportional to a monomial in the variables $u_1,\dots,u_N$.
We can think about the ray $(\fd, H_\fd)$ as being the superposition of rays $(\fd,H_{\fd,k})$. By generic perturbations transverse to their directions, we can separate these rays.
We do such splitting for all the rays of 
$\fD$. 
When two of the perturbed rays meet we are in the situation of elementary scattering, with propagation of the two ingoing rays and emission of a new outgoing ray. 
We iterate the construction until we get a consistent picture. One can show that if the initial perturbations are generic enough, then all the local computations are elementary scatterings. We recover $S(\fD)$ by putting back together all the parallel rays.

When working with perturbed rays, sequences of elementary scatterings producing outgoing rays define balanced graphs in $\R^2$, i.e.\ tropical curves. It is a key point: the combinatorics of the computation of the consistent completion 
$\fD \mapsto S(\fD)$ of a local scattering diagram is the combinatorics of tropical curves in $\R^2$. This is the ultimate explanation for the connection between local scattering diagrams and curve counting. 

\subsection{Curve counting from local scattering diagrams}
In order to obtain a connection with Gromov--Witten theory, we need to specialize the general discussion of local scattering diagrams done previously. We make a particular choice of Lie algebra: we take 
$\fg=\C[M]$, with linear basis given by monomials $z^m$, $m\in M=\Z^2$, and with Lie bracket given by 
\[ [z^{m_1},z^{m_2}]=\det(m_1,m_2)z^{m_1+m_2}\,.\]
Conceptually, viewing $\C[M]$ as the algebra of functions on $(\C^{*})^2$, $[-,-]$ is the Poisson bracket defined by the holomorphic symplectic form $\frac{dx}{x} \wedge \frac{dy}{y}$.

We take $R=\C[\![t]\!]$. Concretely, we apply the formalism of scattering diagrams with $R=\C[t]/t^N$ and we take the limit $N \rightarrow +\infty$. We choose
primitive elements $m_1$ and 
$m_2$ of $M$.
Let $\fD_{m_1,m_2}$ be the local scattering diagram consisting of two ingoing rays 
$(\R_{\geqslant 0}m_1, H_1)$ and 
$(\R_{\geqslant 0}m_2,H_2)$, where 
\[ H_1 \coloneq \sum_{k\geqslant 1}
\frac{(-1)^{k-1}}{k^2} z^{km_1} t^k \]
and
\[ H_2 \coloneq \sum_{k\geqslant 1}
\frac{(-1)^{k-1}}{k^2}z^{km_2} t^k \,.\]
Let $S(\fD_{m_1,m_2})$ be the consistent completion of $\fD_{m_1,m_2}$.
For every $a,b \in \NN$ coprime, let 
$(\fd_{a,b}=-\R_{\geqslant 0}(am_1+bm_2),H_{a,b})$ be the outgoing ray 
of direction $-(am_1+bm_2)$ in $S(\fD_{m_1,m_2})$.

Gross--Pandharipande--Siebert \cite{MR2667135}
have given a Gromov--Witten interpretation of the generating series $H_{a,b}$ computed 
by the local scattering diagram $S(\fD_{m_1,m_2})$. 

Let $Y_{m_1,m_2}^{a,b}$ be the projective toric surface of fan given by the three rays 
$\R_{\geqslant 0}m_1$, $\R_{\geqslant 0}m_2$, and $-\R_{\geqslant 0} (am_1+bm_2)$.
Let $D_1$, $D_2$ and $D_{a,b}$ be the corresponding toric divisors.
Let $X_{m_1,m_2}^{a,b}$ be the projective surface obtained by blowing-up one point on $D_1$ away from $D_1 \cap D_2$ and 
$D_1 \cap D_{a,b}$, and one point on $D_2$ away from $D_2 \cap D_1$ and $D_2 \cap D_{a,b}$. We denote by $E_1$ and $E_2$ the corresponding exceptional divisors. We still denote by $D_1$, $D_2$ and 
$D_{a,b}$ the strict transforms
in $X_{m_1,m_2}^{a,b}$ of 
$D_1$, $D_2$, and $D_{a,b}$.

For every positive integer $k$, there exists a unique class 
$\beta_k \in H_2(X_{m_1,m_2}^{a,b},\Z)$
such that $\beta_k \cdot E_1=ka$, 
$\beta_k \cdot E_2=kb$, and 
$\beta_k \cdot D_{a,b}=k$. 
Let 
$N_{m_1,m_2}^{ka,kb}$ be the Gromov--Witten count of rational curves in 
$X_{m_1,m_2}^{a,b}$ of class $\beta_k$
intersecting $D_{a,b}$ at a single point.
One precise way to define 
$N_{m_1,m_2}^{ka,kb}$ is to use log Gromov--Witten theory of 
$X_{m_1,m_2}^{a,b}$ relatively to the 
divisor 
$D_1 \cup D_2 \cup D_{a,b}$.

It seems that we are using a different surface $X_{m_1,m_2}^{a,b}$ for each choice of $a$ and $b$. In fact, we can replace 
$Y_{m_1,m_2}^{a,b}$ by any projective toric surface whose fan contains the rays $\R_{\geqslant 0}m_1$ and
$\R_{\geqslant 0}m_2$, and then, for every 
$a$ and $b$, we can interpret $(ka,kb)$ as a well-defined relative condition in log Gromov--Witten theory and define the log Gromov--Witten invariants $N_{m_1,m_2}^{ka,kb}$. 
The log Gromov--Witten invariants are independent of the precise choice of toric surface by invariance of log Gromov--Witten invariants under log birational modifications \cite{abramovich2013invariance}.

The main result of Gross--Pandharipande--Siebert \cite{MR2667135} is then:

\begin{thm} \label{thm_gps}
For every $a,b \in \NN$ coprime, the generating series $H_{a,b}$ attached to the ray of direction $-(am_1+bm_2)$ in the local scattering diagram $S(\fD_{m_1,m_2})$ is given by 
\[ H_{a,b}=\sum_{k \geqslant 1} 
N_{m_1,m_2}^{ka,kb} z^{k(am_1+bm_2)} t^{k(a+b)} \,.\]
\end{thm}

\begin{proof}
We present a sketch of the proof given in \cite{MR2667135}.
After sending the blown-up points ``at infinity"
by a degeneration, the computation of the log Gromov--Witten invariants $N_{m_1,m_2}^{ka,kb}$ of 
$X_{m_1,m_2}^{a,b}$ can be reduced to the
computation of log Gromov--Witten invariants of $Y_{m_1,m_2}^{a,b}$ with contact conditions along $D_1$, $D_2$, and with only a single intersection point with $D_{a,b}$. 
The factors $\frac{(-1)^{k_1}}{k^2}$
in $H_1$ and $H_2$ come from relative Gromov--Witten invariants of $\PP^1$ appearing in the degeneration argument.

The log Gromov--Witten invariants of the toric surface 
$Y_{m_1,m_2}^{a,b}$ can be computed in terms of enumeration of tropical curves in $\R^2$
\cite{MR2137980, MR2259922}: one constructs by toric means appropriate normal crossing degenerations of 
$Y_{m_1,_2}^{a,b}$ and the tropical curves appear as dual intersection graphs of the degenerated curves in the special fiber.

It remains to use the correspondence between scattering diagrams and tropical curves sketched as the end of  \cref{section_local_scattering}.
\end{proof}

\subsection{Scattering diagrams}
\label{section_scattering_diagrams}

The divisor 
\[D_1 \cup D_2 \cup D_{a,b}\] 
is anticanonical on the surface $X_{m_1,m_2}^{a,b}$.
In other words, the pair 
\[(X_{m_1,m_2}^{a,b}, D_1 \cup D_2 \cup D_{a,b})\] is a log Calabi-Yau surface. Theorem \ref{thm_gps} computes a class of log Gromov--Witten invariants of $(X_{m_1,m_2}^{a,b}, D_1 \cup D_2 \cup D_{a,b})$ in terms of a local scattering diagram. More generally, for every 
log Calabi-Yau surface $(Y,D)$ with $D$ a cycle of
rational curves, there is a version of Theorem \ref{thm_gps} computing log Gromov--Witten invariants for rational curves intersecting $D$ at a single point in terms of a local scattering diagram.

We are interested in log Gromov--Witten invariants for rational curves in 
$(\PP^2,E)$ intersecting $E$ at a single point. The pair $(\PP^2,E)$ is a log Calabi-Yau surface but $E$ is a smooth genus-$1$ curve and not a cycle of rational curves.
In particular, we cannot use a local scattering diagram to compute the 
invariants $\Omega_{d,k}^{\PP^2/E}$.
The invariants $\Omega_{d,k}^{\PP^2/E}$ will be computed using a (global, not local) scattering diagram constructed from a normal crossing degeneration of $(\PP^2,E)$ \cite{gabele2019tropical}. 

Let $B_0$ be an integral affine manifold. A \emph{scattering diagram} $\fD$ on $B_0$ is a collection
of rays $(\fd, H_{\fd})$ on $B_0$ such that, locally
near each point $b \in B_0$, we see a local scattering diagram $\fD_b$ in the sense of \ref{section_local_scattering}.
We refer to \cite{MR2846484} and \cite{bousseau2019scattering}
for more precise definitions, dealing in particular with convergence issues. A scattering diagram $\fD$ is said to be \emph{consistent} if all the local scattering diagrams 
$\fD_b$ are consistent in the sense of \cref{section_local_scattering}. Given a scattering diagram 
$\fD$ on $B_0$, there is a canonical way to produce a consistent scattering diagram $S(\fD)$. When some rays 
intersect at some point, we apply Theorem \ref{KS_lemma}
and add some new rays to guarantee local consistency around this point. Then, we propagate the new rays and we iterate the construction.

\section{Scattering diagrams as a bridge between sheaves and curves}

\subsection{Scattering diagram from relative Gromov--Witten theory}
\label{section_scattering_gw}
In this section, we follow the work of Gabele \cite{gabele2019tropical}.
Let 
\[ \cX \coloneq \{ ([x:y:z:w], t) \in \PP(1,1,1,3) \times 
\A^1 |\, xyz-t^3(x^3+y^3+z^3+w)=0\}\,.\]
Denote by $\cX_t$ the fiber over $t \in \A^1$. The hypersurface
$\cX_t$ in $\PP(1,1,1,3)$ intersects the toric boundary divisor
$\PP^2=\{ w=0\}$ along the cubic curve 
\[ E_t: xyz-t^3(x^3+y^3+z^3)=0\,.\] 
The special fiber $\cX_0$ breaks into the union of the three other toric divisors 
$\{x=0\}$, $\{y=0\}$, $\{z=0\}$, each one being isomorphic to the weighted projective plane 
$\PP(1,1,3)$. The cubic $E_0$
breaks into a triangle of lines.

Let $(B, \mathscr{P})$ be the dual intersection complex of 
$\cX_0$. The polyhedral decomposition 
$\cP$ contains 3 vertices $v_1$, $v_2$, $v_3$, dual to the 3 components of $\cX_0$, defining a triangle $T$ dual to the triple intersection point of the 3 components of $\cX_0$. 
As each irreducible component of $\cX_0$ is toric, there is a natural way to define an integral affine structure on the complement $B_0$ in $B$ of 3 focus-focus singularities $x_1$, $x_2$, $x_3$. 
The 3 singularities of the affine structure are related to the fact that the total space of $\cX$ has 3 nodal points and that the family $\cX \rightarrow \A^1$
is not log smooth at these points.

We define a scattering diagram $\fD_{\PP^2/E}$ on $B_0$
consisting of 6 rays emanating from the 3 singularities 
in the monodromy invariant directions defined by the edges of $T$, and with attached functions
\[ H = \sum_{k\geqslant 1}
\frac{(-1)^{k-1}}{k^2} z^{km}\,, \]
where $m$ is the direction of the ray pointing towards the singularity. Let $S(\fD_{\PP^2/E})$ be the consistent scattering diagram on $B_0$ obtained by consistent completion of $\fD_{\PP^2/E}$.

Figure: $(B, \mathscr{P})$. The two unbounded half-lines meeting at each singularity $x_i$ need to be identified and the affine structure needs to be glued across this identification by an explicit transformation in $SL(2,\Z)$.
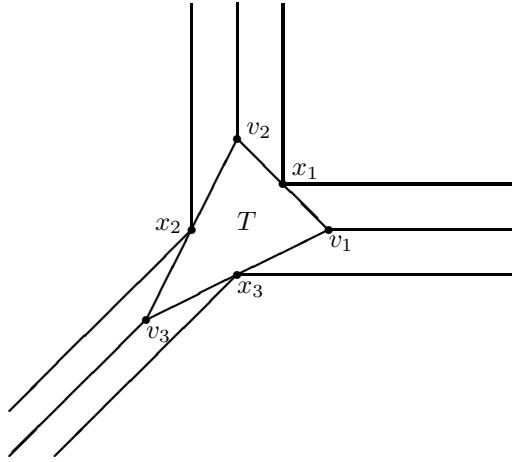
\begin{center}
\setlength{\unitlength}{1.2cm}
\begin{picture}(10,5)
\thicklines
\put(4,1.5){\circle*{0.1}}
\put(4,1.3){$v_3$}
\put(5,3.5){\circle*{0.1}}
\put(5.1,3.55){$v_2$}
\put(6,2.5){\circle*{0.1}}
\put(6,2.3){$v_1$}
\put(5.5,3){\circle*{0.1}}
\put(5.6,3.1){$x_1$}
\put(5.5,3){\line(1,0){2.5}}
\put(5.5,3){\line(0,1){2}}
\put(5,2){\circle*{0.1}}
\put(5,1.8){$x_3$}
\put(5,2){\line(1,0){3}}
\put(5,2){\line(-1,-1){2}}
\put(4.5,2.5){\circle*{0.1}}
\put(4.1,2.5){$x_2$}
\put(4.5,2.5){\line(-1,-1){2}}
\put(4.5,2.5){\line(0,1){2.5}}
\put(4,1.5){\line(-1,-1){1.5}}
\put(5,3.5){\line(0,1){1.5}}
\put(6,2.5){\line(1,0){2}}
\put(4,1.5){\line(2,1){2}}
\put(4,1.5){\line(1,2){1}}
\put(5,3.5){\line(1,-1){1}}
\put(5,2.5){$T$}
\end{picture}
\end{center}

 Let $\widetilde{B-T}$ be the universal cover of the complement in $B$
of the triangle $T$. One can identify $\widetilde{B-T}$
(modulo a rescaling by $3$ of $y$) as integral affine manifold with an open subset of 
\[ U \coloneq 
\{  (x,y) \in \R^2 |\, y > -\frac{x^2}{2}\}\,.\] 
More precisely, $\widetilde{B-T}$ is identified with 
\[ \{ (x,y)\in \R^2 |\, y>f(x) \}\,,\]
where $f(x)$ is a continuous piecewise linear function
approximating $-\frac{x^2}{2}$.
The singularities 
$x_1,x_2,x_3$ on the boundary of $B-T$ lift to the points 
$(n,-\frac{n^2}{2})$, $n \in \Z$, all on the boundary of $U$
given by the parabola of equation 
$y=-\frac{x^2}{2}$. The monodromy invariants directions at the singularities lift to the tangent lines to the parabola at the points $(n,-\frac{n^2}{2})$. One can show that the rays of the scattering diagram
$S(\fD_{\PP^2/E})$ never enter the interior of the triangle 
$T$. Thus, we can consider the lift 
$\tilde{S}(\fD_{\PP^2/E})$ of $S(\fD_{\PP^2/E})$ to $U$,
see Figure \ref{figure_scattering}.

We claim that the scattering diagram $\tilde{S}(\fD_{\PP^2/E})$ computes the Gromov--Witten invariants $N_{0,d}^{\PP^2/E,k}$ introduced in
\cref{section_scattering_gw}. 

We can show that 
vertical asymptotic rays in $\tilde{S}(\fD_{\PP^2/E})$
are contained in vertical lines of equation $x=x_0$ with $x_0 \in \Q$. Given $x_0 \in \Q$, we denote by $[x_0]$ its image in 
$\Q/3\Z$.
For every $G$ an abelian group and $x$ an element of $G$ of finite order divisible by $3$, we denote by $d(x)$ the smallest positive integer such that 
$(3d(x))x=0$ in $G$.
For every $\ell \in \Z_{\geqslant 1}$, we denote by $r_\ell$ the number of elements $x \in \Z/(3\ell)$ such that $d(x)=\ell$.
For every $k, \ell \in \Z_{\geqslant 1}$, we denote by
$s_{k,\ell}$ the number of $x=(a,b)
\in \Z/(3k) \times \Z/(3k)$
such that 
$d(x)=k$ and $d(a)=\ell$.

\begin{thm}\cite{gabele2019tropical} \label{tropical_correspondence}
The function attached to an asymptotic vertical ray in 
$\tilde{S}(\fD_{\PP^2/E})$ of equation $x=x_0$ with $d([x_0])=\ell$ is 
\[ H_\ell = \sum_{\substack{k \geqslant 1 
\\\ell |k}} 
\sum_{\substack{d \geqslant 1 
\\k |d}} \frac{s_{k\ell}}{r_\ell} 
N_{0,d}^{\PP^2/E,k} z^{(0,d)}\,.\]
\end{thm}

Theorem \ref{tropical_correspondence} is proved by Gabele 
\cite{gabele2019tropical}. It is a general expectation in the Gross-Siebert approach to mirror symmetry that scattering diagrams should encode enumeration of holomorphic disks \cite{MR2846484, cps}.

\begin{proof}
We give a sketch of proof of Theorem \ref{tropical_correspondence}. We are interested in the log Gromov--Witten invariants $N_{0,d}^{\PP^2/E,k}$ of 
$(\PP^2,E)$. By deformation invariance of log Gromov--Witten theory, we can compute the invariants $N_{0,d}^{\PP^2/E,k}$ of $(\PP^2,E)$ on the special fiber $\cX_0$. According to the decomposition formula of \cite{abramovich2017decomposition}, we can decompose 
$N_{0,d}^{\PP^2/E,k}$ into pieces indexed by rigid tropical curves in $B_0$. 

One important point is that we can follow the torsion points of $E$ is the degeneration and in the tropicalization. Indeed, for every positive integer $n$, up to doing a base change and some blow-ups, we can consider a new degeneration where the elliptic curve breaks into a cycle of $n$ rational components, and such that the $n$-torsion points are monodromy invariant. The $n^2$ $n$-torsion points degenerate into $n$ points on each of the $n$ components of the cycle. 
Tropically, the family of elliptic curves defines the circle ``at infinity of $B$" and the $n$ components of the cycle correspond to the $n$ $n$-torsion points of this circle.
Thus, we cannot distinguish tropically the $n^2$ $n$-torsion points of the elliptic curve, but we can see $n$ packets of $n$-torsion points and it is enough for us. Indeed, we already know by Lemma \ref{lem_monodromy}
that the invariants $N_{0,d}^{\PP^2/E,p}$
depends on $p$ only through $d(p)$ and so we do not have 
$(3d)^2$ but only $3d$ unknowns. The factor $\frac{s_{k\ell}}{r_\ell}$ in Theorem \ref{tropical_correspondence} comes from the comparison 
between torsion points of the elliptic curves and torsion points of the tropical elliptic curve.

In order to evaluate the contribution of a tropical curve to $N_{0,d}^{\PP^2/E,k}$, we include this tropical curve in a refinement of the polyhedral decomposition $\cP$. This defines a new family in which the components of the stable log maps of interest maps transversely into the components of the new special fiber. Each component is a toric surface
and so it follows from Theorem \ref{thm_gps} that the local scatterings in $S(\fD_{\PP^2/E})$ correspond to counts of rational curves in the toric components of the special fiber. It remains to glue these local contributions to conclude.
\end{proof}

According to Theorem \ref{tropical_correspondence}, 
the Gromov--Witten invariants 
$N_{0,d}^{\PP^2/E,k}$, or equivalently the BPS 
invariants $\Omega_{d,k}^{\PP^2/E}$, can be computed from the scattering diagram $\tilde{S}(\fD_{\PP^2/E})$, which has a purely algebraic/algorithmic definition. Thus, one can translate Theorem \ref{thm_gw-bps} into a purely algebraic statement about $\tilde{S}(\fD_{\PP^2/E})$. One might hope to give a purely algebraic proof of this statement. Unfortunately, such a proof is not known: 
$\tilde{S}(\fD_{\PP^2/E})$ is a quite complicated object, see Figure \ref{figure_scattering}. 
To make progress, we need to come back to geometry
(but not the same geometry we started with...).
In the next sections, we explain how $\tilde{S}(\fD_{\PP^2})$
appears in the context of stability conditions on 
$\D^b(\PP^2)$ and how Theorem \ref{thm_gw-bps} translates into Theorem \ref{thm_sheaves}, thus proving Theorem
\ref{thm_main}.

\subsection{Scattering diagram from stability conditions}
\label{section_scattering_stability}

The main idea is to embed $U$ in the space 
$\Stab \D^b(\PP^2)$ of Bridgeland stability conditions on the derived category $\D^b(\PP^2)$
of coherent sheaves on $\PP^2$
and to give a description of the scattering diagram 
$\tilde{S}(\fD_{\PP^2/E})$.

We have 
\[ K_0(\PP^2) \simeq \Gamma \coloneq \Z^3\]
\[ [F] \mapsto \gamma(F)
=(r(F),d(F),\chi(F))\,,\]
where $r$ is the rank, $d$ is the degree and $\chi$ the Euler characteristic.
Recall that a Bridgeland stability condition
\cite{MR2373143}
on $\D^b(\PP^2)$
is a pair $\sigma=(\cA,Z)$, where
$\cA \subset \D^b(\PP^2)$
is an abelian category, heart of a bounded t-structure on $\D^b(\PP^2)$, and 
\[ Z \colon \Gamma \rightarrow \C \]
\[ \gamma \mapsto Z_\gamma\]
is a linear map, called the central charge, such that:
\begin{enumerate}
\item 
For every object $F \neq 0$ in $\cA$, we have 
$Z_{\gamma(F)} \in \{ z \in \C |\,\Ima z>0\} \cup \R_{<0}$. 
Thus, for every $F \neq 0$ in $\cA$, we can define 
$\phi(F) \coloneq \frac{1}{\pi} 
\Arg Z_{\gamma(F)} \in (0,1]$. 
We say that an object $F \neq 0$ in $\cA$ is 
$\sigma$-semistable if for every $F' \neq 0$
subobject of $F$ in $\cA$, we have 
$\phi(F') \leqslant \phi(F)$.
\item Every  object $F \neq 0$ in $\cA$ admits a Harder-Narasimhan filtration \[0=F_0 \subset F_1
\subset \cdots \subset F_n=F\]
in $\cA$,
whose factors $G_i\coloneq F_i/F_{i-1}$ are $\sigma$-semistable objects in $\cA$ with 
\[\phi(F_1)>\phi(F_2)>\cdots>\phi(F_n)\,.\]
\item Support property:
there exists a quadratic form $Q$ on the $\R$-vector space $\Gamma \otimes \R$
    such that the kernel of $Z$ in $\Gamma \otimes \R$ is negative
        definite with respect to $Q$,
and for every $\sigma$-semistable object $F$, we have $Q(\gamma(F)) \geqslant 0$.
\end{enumerate}
According to \cite{MR2373143}, the 
space $\Stab \D^b(\PP^2)$ of Bridgeland stability conditions on 
$\D^b(\PP^2)$ has a natural structure of complex manifold of complex dimension $3$.

\begin{lem}\cite{bousseau2019scattering}
\label{lem_stability}
There exists an embedding 
\[ U \hookrightarrow \Stab \D^b(\PP^2) \]
\[ (x,y) \mapsto \sigma^{(x,y)}
=(\cA^{(x,y)}, Z^{(x,y)})\]
such that, for every $\gamma
=(r,d,\chi) \in \Gamma$, we have 
\[ Z_\gamma^{(x,y)}
=ry+dx+r+\frac{3d}{2}-\chi+i(d-rx)
\sqrt{x^2+2y} \,.\]
\end{lem}

\begin{proof}
According to 
\cite{MR2376815, MR2998828, MR2852118},
there exists an embedding 
\[ \mathbb{H}
\coloneq 
\{ (s,t)\in \R^2|\, t>0\} \hookrightarrow \Stab \D^b(\PP^2)\]
\[ (s,t) \mapsto \sigma^{(s,t)}
=(\cA^{(s,t)},Z^{(s,t)})\]
such that
\[Z_{\gamma(E)}^{(s,t)}=-\int_{\PP^2}
e^{-(s+it)H} \ch(E)\,,\]
where $H \coloneq c_1(\cO(1))$.
We obtain the desired embedding via 
the quadratic change of variables 
\[ \mathbb{H} \rightarrow U \]
\[(s,t) \mapsto  (x,y)=\left(s,-\frac{1}{2}(s^2-t^2)\right)\,.\]
\end{proof}

From now on, we use Lemma 
\ref{lem_stability} to view $U$ as a subset of $\Stab \D^b(\PP^2)$. 

For every $\sigma \in U$ and 
$\gamma \in \Gamma$, we have a moduli space 
$M_\gamma^\sigma$ parametrizing $S$-equivalence classes of $\sigma$-semistable objects $F$ with $\gamma(F)=\gamma$.
Given $\gamma \in \Gamma$, there are finitely many real codimension $1$ loci in $U$, called walls, in the complement of which $M_\gamma^\sigma$ is a
constant function of $\sigma$, and across which $M_\gamma^\sigma$ jumps.
Given $\gamma \in \Gamma$ and 
$x \in \R$, we can show that, for every 
$y \in \R_{>0}$ large enough, the moduli space $M_\gamma^{\sigma=(x,y)}$ coincides with the moduli space of Gieseker semistable sheaves of class $\gamma$.

For every $\sigma \in U$
and $\gamma \in \Gamma$, we denote by 
$Ie(M_\gamma^\sigma)$ the Euler characteristic of $M_\gamma^\sigma$ for the 
intersection cohomology and we define
\[ \Omega_\gamma^\sigma 
\coloneq (-1)^{\dim M_\gamma^\sigma} 
Ie(M_\gamma^\sigma) \in \Z \,.\]
The invariants $\Omega_\gamma^\sigma$ jump across the walls. We can show 
\cite{bousseau2019scattering} 
that the invariants 
$\Omega_\gamma^\sigma$ are Donaldson-Thomas invariants of the noncompact Calabi-Yau 3-fold $K_{\PP^2}$, total space of the canonical line bundle of $\PP^2$, and that their jumps across the walls are described by the Kontsevich-Soibelman wall-crossing formula \cite{kontsevich2008stability}. A key technical tool in this proof is a $\Ext^2$ vanishing result for $\sigma$-semistable objects due to Li-Zhao \cite{MR3936077}.

We use the invariants $\Omega_\gamma^\sigma$ to define a scattering diagram $\fD_{\PP^2}$ on $U$ as follows.
The rays of $\fD_{\PP^2}$ are indexed by $\gamma \in \Gamma$
and given by 
\[ R_\gamma \coloneq \{ \sigma=(x,y)\in U |\, \Rea Z_\gamma^\sigma=0\,, \Omega_\gamma^\sigma \neq 0 \}\,.\]
As we have $\Rea Z_\gamma^{(x,y)}=ry+dx+r+\frac{3d}{2}-\chi$, the locus $R_\gamma$ is indeed a straight line in $U$, of direction 
$(-r,d)$. To each segment of $R_\gamma$ on which the invariants
$\Omega_{k\gamma}^\sigma$ do not jump, we attach the function 
\[ \sum_{k \geqslant 1} \frac{\Omega_{k\gamma}^\sigma}{k^2}
z^{(kr,-kd)}\,.\]

\begin{thm}\cite{bousseau2019scattering}
The scattering diagram $\fD_{\PP^2}$ is consistent.
\end{thm}

\begin{proof}
When two rays $R_{\gamma_1}$ and $R_{\gamma_2}$
intersect at a point $\sigma \in U$, we have by definition 
$\Rea Z_{\gamma_1}^\sigma=0$ and $\Rea Z_{\gamma_2}^\sigma=0$.
In particular, the central charges $Z_{\gamma_1}^\sigma$ and $Z_{\gamma_2}^{\sigma}$ are collinear and so $\sigma$ is on a
(potential) wall. One checks that the consistency of the local scattering diagram around $\sigma$ is a consequence of the Kontsevich-Soibelman wall-crossing formula describing the jumps of the invariants across the wall. 
\end{proof}

\subsection{Comparison of the scattering diagrams}
In \cref{section_scattering_gw} we defined a scattering diagram 
$\tilde{S}(\fD_{\PP^2})$ on $U$, describing tropically 
log Gromov--Witten invariants in a normal crossing degeneration of the pair $(\PP^2,E)$.
On the other hand, we defined in  \cref{section_scattering_stability} another scattering diagram $\fD_{\PP^2}$ on $U$, 
describing wall-crossing behavior of counting invariants of the derived category $\D^b(\PP^2)$.

\begin{thm}\cite{bousseau2019takahashi} \label{thm_scattering_equality}
We have $\tilde{S}(\fD_{\PP^2/E})=\fD_{\PP^2}$.
\end{thm}

\begin{proof}
We give a sketch of the proof. We know that both 
$\tilde{S}(\fD_{\PP^2/E})$ and $\fD_{\PP^2}$ are consistent scattering diagrams on $U$. In order to prove that they coincide, it is enough to show that they have the same initial data. Initial data for $\tilde{S}(\fD_{\PP^2/E})$ 
are rays emitted by the singular points 
$(n,-\frac{n^2}{2})$ and tangent to the parabola 
$y=-\frac{x^2}{2}$. On the side of 
$\fD_{\PP^2}$, one can identify these rays with the rays
$R_{\gamma(\cO(n))}$ defined by the line bundles $\cO(n)$
(and their shift $\cO(n)[1]$). In particular, the singular points $(n,-\frac{n^2}{2})$ are exactly the points where the central charge $Z_{\gamma(\cO(n))}$ goes to zero.
To conclude, one needs to show that the rays $R_{\gamma(\cO(n))}$ are the only rays in $\fD_{\PP^2}$
existing in a small neighborhood of the parabola 
$y=-\frac{x^2}{2}$. This follows from a description of the stability conditions near the parabola 
$y=-\frac{x^2}{2}$ in terms of quiver representations.
\end{proof}

We use Theorem \ref{thm_scattering_equality}
to obtain a comparison of the relative Gromov--Witten invariants $N_{0,d}^{\PP^2/E}$ and of the dimension $1$ sheaves invariants $\Omega_{d,\chi}^{\PP^2}$.

For every $d \in \Z_{>0}$ and $\chi \in \Z$, we define
\[ \ell_{d,\chi} \coloneq \frac{d}{\gcd(d,\chi)} \in  \Z_{>0}\,.\]

\begin{thm} \label{thm_gw_sheaves}
For every $d\in \Z_{>0}$ and $\chi \in \Z$, we have 
\[ 
\Omega_{d,\chi}^{\PP^2}=\sum_{\ell_{d,\chi}|k|d}
\frac{s_{k,\ell_{d,\chi}}}{r_{\ell_{d,\chi}}} \Omega_{d,k}^{\PP^2/E}\,.\]
\end{thm}

\begin{proof}
According to Theorem \ref{tropical_correspondence}, the asymptotic vertical rays of $\tilde{S}(\fD_{\PP^2/E})$
compute the relative Gromov--Witten invariants 
$N_{0,d}^{\PP^2/E}$. On the other hand, the asymptotic vertical rays of $\fD_{\PP^2}$ are defined in terms of the invariants $\Omega_{d,\chi}^{\PP^2}$. Indeed, vertical rays correspond to classes $\gamma=(r,d,\chi)$ with $r=0$, i.e.\ sheaves of dimension $1$, and $\sigma$-semistability coincides with Gieseker semistability for $\sigma=(x,y)\in U$
with $y>>0$. The result follows from the equality of scattering diagrams given by Theorem 
\ref{thm_scattering_equality}.
\end{proof}

One should view the sheaf/Gromov--Witten correspondence given by Theorem \ref{thm_gw_sheaves} as an analogue of the correspondence presented in \cite{MR2662867} between quiver Donaldson-Thomas invariants and Gromov--Witten invariants of log Calabi-Yau surfaces $(Y,D)$ with $D$ a cycle of rational curves. 
The analogy also holds at the level of proofs: in both cases, a scattering diagram is used as an intermediate algebraic/combinatorial object between two different looking geometries. 
The main difference is that the scattering diagram of \cite{MR2662867} is a local scattering diagram
(in the sense of \cref{section_local_scattering}), whereas we consider a scattering diagram containing infinitely many such local scattering diagrams. 
Equivalently, the quiver Donaldson-Thomas invariants of \cite{MR2662867} 
involve a fixed abelian category, 
whereas we are crucially working with stability conditions
with moving abelian hearts
on the triangulated category $\D^b(\PP^2)$.

Our main result, Theorem \ref{thm_main}, follows directly from Theorem \ref{thm_gw_sheaves}.

\newpage

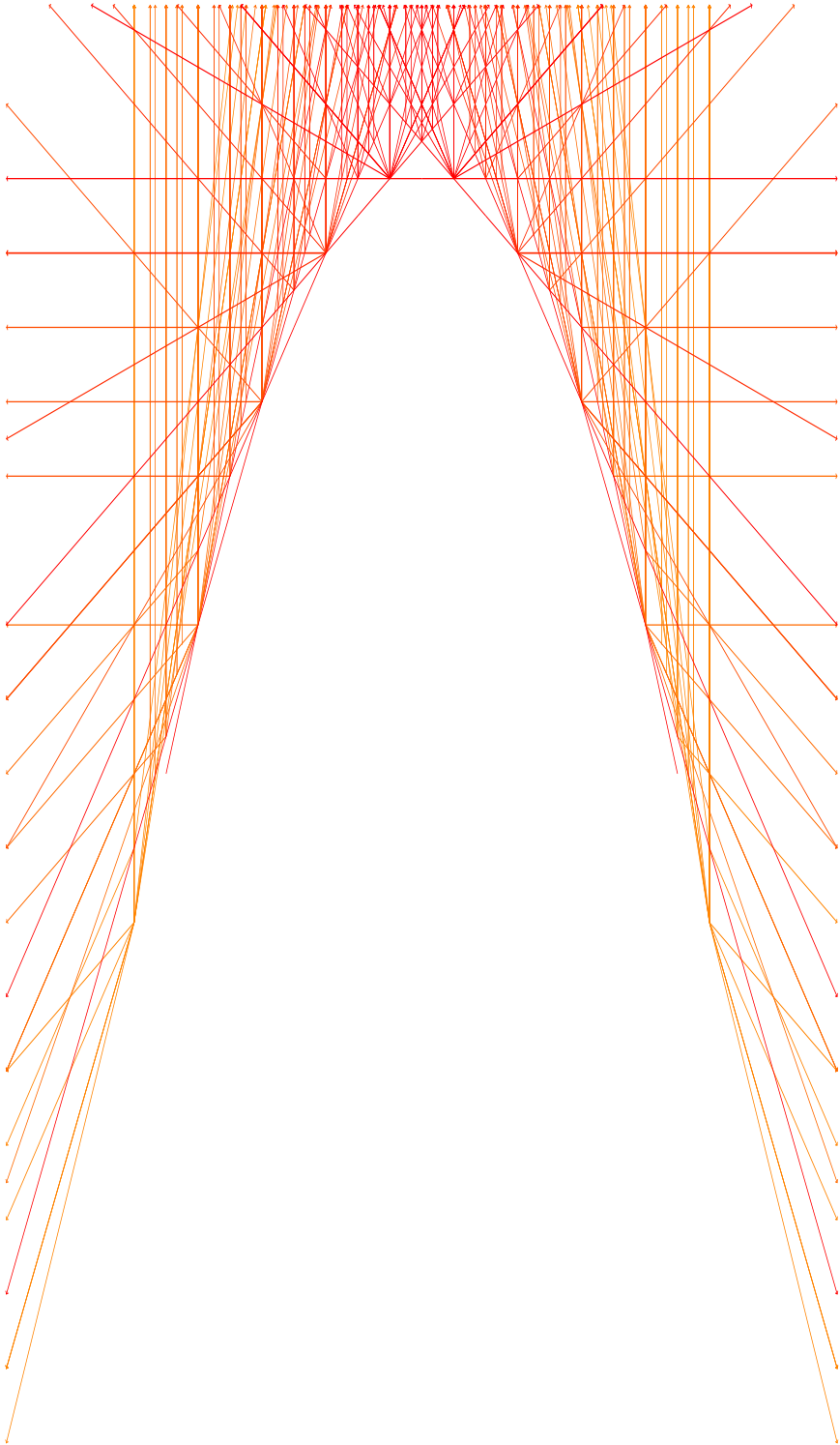
\begin{figure}[h!]
\centering
\resizebox{0.90\textwidth}{0.90\textheight}{
\rotatebox{90}{
\begin{tikzpicture}[xscale=0.6,yscale=2.7,font=\fontsize{6}{6},define rgb/.code={\definecolor{mycolor}{RGB}{#1}}, rgb color/.style={define rgb={#1},mycolor}]
\draw[->,rgb color={255,132,0}] (-30.0,-4.00) -- (-51.0,-6.00);
\draw[->,rgb color={255,132,0}] (-30.0,-4.00) -- (-48.0,-6.00);
\draw[->,rgb color={255,132,0}] (-30.0,-4.00) -- (-48.0,-6.00);
\draw[->,rgb color={255,132,0}] (-30.0,-4.00) -- (-48.0,-6.00);
\draw[->,rgb color={255,132,0}] (-30.0,-4.00) -- (-42.0,-6.00);
\draw[->,rgb color={255,132,0}] (-30.0,-4.00) -- (-36.0,-6.00);
\draw[->,rgb color={255,132,0}] (-30.0,-4.00) -- (7.00,-4.00) node[right]{};
\draw[->,rgb color={255,132,0}] (-30.0,-4.00) -- (7.00,-4.00) node[right]{};
\draw[->,rgb color={255,132,0}] (-30.0,-4.00) -- (7.00,-4.00) node[right]{};
\draw[->,rgb color={255,132,0}] (-30.0,-4.00) -- (7.00,-4.00) node[right]{};
\draw[->,rgb color={255,132,0}] (-30.0,-4.00) -- (7.00,-4.00) node[right]{};
\draw[->,rgb color={255,132,0}] (-30.0,-4.00) -- (7.00,-4.00) node[right]{};
\draw[->,rgb color={255,132,0}] (-27.0,-4.00) -- (7.00,-4.00) node[right]{};
\draw[->,rgb color={255,132,0}] (-27.0,-4.00) -- (7.00,-4.00) node[right]{};
\draw[->,rgb color={255,132,0}] (-30.0,-4.00) -- (7.00,-4.00) node[right]{};
\draw[->,rgb color={255,132,0}] (-30.0,-4.00) -- (7.00,-1.94);
\draw[->,rgb color={255,132,0}] (-30.0,-4.00) -- (7.00,-1.94);
\draw[->,rgb color={255,132,0}] (-30.0,-4.00) -- (7.00,-1.94);
\draw[->,rgb color={255,132,0}] (-30.0,-4.00) -- (7.00,-2.24);
\draw[->,rgb color={255,132,0}] (-30.0,-4.00) -- (7.00,-2.46);
\draw[->,rgb color={255,132,0}] (-30.0,-4.00) -- (7.00,-1.76);
\draw[->,rgb color={255,132,0}] (-30.0,5.00) -- (-51.0,7.00);
\draw[->,rgb color={255,132,0}] (-30.0,5.00) -- (-48.0,7.00);
\draw[->,rgb color={255,132,0}] (-30.0,5.00) -- (-48.0,7.00);
\draw[->,rgb color={255,132,0}] (-30.0,5.00) -- (-48.0,7.00);
\draw[->,rgb color={255,132,0}] (-30.0,5.00) -- (-42.0,7.00);
\draw[->,rgb color={255,132,0}] (-30.0,5.00) -- (-36.0,7.00);
\draw[->,rgb color={255,132,0}] (-30.0,5.00) -- (7.00,5.00) node[right]{};
\draw[->,rgb color={255,132,0}] (-30.0,5.00) -- (7.00,5.00) node[right]{};
\draw[->,rgb color={255,132,0}] (-30.0,5.00) -- (7.00,5.00) node[right]{};
\draw[->,rgb color={255,132,0}] (-30.0,5.00) -- (7.00,5.00) node[right]{};
\draw[->,rgb color={255,132,0}] (-30.0,5.00) -- (7.00,5.00) node[right]{};
\draw[->,rgb color={255,132,0}] (-30.0,5.00) -- (7.00,5.00) node[right]{};
\draw[->,rgb color={255,132,0}] (-27.0,5.00) -- (7.00,5.00) node[right]{};
\draw[->,rgb color={255,132,0}] (-27.0,5.00) -- (7.00,5.00) node[right]{};
\draw[->,rgb color={255,132,0}] (-30.0,5.00) -- (7.00,5.00) node[right]{};
\draw[->,rgb color={255,132,0}] (-30.0,5.00) -- (7.00,2.94);
\draw[->,rgb color={255,132,0}] (-30.0,5.00) -- (7.00,2.94);
\draw[->,rgb color={255,132,0}] (-30.0,5.00) -- (7.00,2.94);
\draw[->,rgb color={255,132,0}] (-30.0,5.00) -- (7.00,3.24);
\draw[->,rgb color={255,132,0}] (-30.0,5.00) -- (7.00,3.46);
\draw[->,rgb color={255,132,0}] (-30.0,5.00) -- (7.00,2.76);
\draw[->,rgb color={255,132,0}] (-27.0,-4.00) -- (-39.0,-6.00);
\draw[->,rgb color={255,132,0}] (-27.0,-4.00) -- (7.00,-2.38);
\draw[->,rgb color={255,132,0}] (-27.0,5.00) -- (-39.0,7.00);
\draw[->,rgb color={255,132,0}] (-27.0,5.00) -- (7.00,3.38);
\draw[->,rgb color={255,132,0}] (-22.5,-3.75) -- (7.00,-3.75) node[right]{};
\draw[->,rgb color={255,132,0}] (-24.8,-3.75) -- (7.00,-3.75) node[right]{};
\draw[->,rgb color={255,107,0}] (-22.5,4.75) -- (7.00,4.75) node[right]{};
\draw[->,rgb color={255,107,0}] (-24.8,4.75) -- (7.00,4.75) node[right]{};
\draw[->,rgb color={255,132,0}] (-24.0,-3.67) -- (7.00,-3.67) node[right]{};
\draw[->,rgb color={255,132,0}] (-24.0,-3.67) -- (7.00,-3.67) node[right]{};
\draw[->,rgb color={255,0,0}] (-24.0,-3.50) -- (7.00,-0.917);
\draw[->,rgb color={255,0,0}] (-24.0,4.50) -- (7.00,1.92);
\draw[->,rgb color={255,107,0}] (-24.0,4.67) -- (7.00,4.67) node[right]{};
\draw[->,rgb color={255,132,0}] (-22.5,-3.50) -- (-30.0,-6.00);
\draw[->,rgb color={255,132,0}] (-22.5,-3.50) -- (7.00,-3.50) node[right]{};
\draw[->,rgb color={255,132,0}] (-22.5,-3.50) -- (7.00,-3.50) node[right]{};
\draw[->,rgb color={255,132,0}] (-21.0,-3.50) -- (7.00,-3.50) node[right]{};
\draw[->,rgb color={255,132,0}] (-18.0,-3.50) -- (7.00,-3.50) node[right]{};
\draw[->,rgb color={255,132,0}] (-22.5,-3.50) -- (7.00,-3.50) node[right]{};
\draw[->,rgb color={255,132,0}] (-22.5,-3.50) -- (7.00,-2.10);
\draw[->,rgb color={255,107,0}] (-22.5,4.50) -- (-30.0,7.00);
\draw[->,rgb color={255,107,0}] (-22.5,4.50) -- (7.00,4.50) node[right]{};
\draw[->,rgb color={255,107,0}] (-22.5,4.50) -- (7.00,4.50) node[right]{};
\draw[->,rgb color={255,107,0}] (-21.0,4.50) -- (7.00,4.50) node[right]{};
\draw[->,rgb color={255,107,0}] (-18.0,4.50) -- (7.00,4.50) node[right]{};
\draw[->,rgb color={255,107,0}] (-22.5,4.50) -- (7.00,4.50) node[right]{};
\draw[->,rgb color={255,107,0}] (-22.5,4.50) -- (7.00,3.10);
\draw[->,rgb color={255,132,0}] (-20.0,-3.33) -- (7.00,-3.33) node[right]{};
\draw[->,rgb color={255,107,0}] (-20.0,4.33) -- (7.00,4.33) node[right]{};
\draw[->,rgb color={255,107,0}] (-20.0,4.33) -- (7.00,4.33) node[right]{};
\draw[->,rgb color={255,132,0}] (-16.5,-3.25) -- (7.00,-3.25) node[right]{};
\draw[->,rgb color={255,132,0}] (-18.8,-3.25) -- (7.00,-3.25) node[right]{};
\draw[->,rgb color={255,107,0}] (-16.5,4.25) -- (7.00,4.25) node[right]{};
\draw[->,rgb color={255,107,0}] (-18.8,4.25) -- (7.00,4.25) node[right]{};
\draw[->,rgb color={255,107,0}] (-18.0,-3.00) -- (-40.5,-6.00);
\draw[->,rgb color={255,107,0}] (-18.0,-3.00) -- (-36.0,-6.00);
\draw[->,rgb color={255,107,0}] (-18.0,-3.00) -- (-36.0,-6.00);
\draw[->,rgb color={255,107,0}] (-18.0,-3.00) -- (-36.0,-6.00);
\draw[->,rgb color={255,107,0}] (-18.0,-3.00) -- (-27.0,-6.00);
\draw[->,rgb color={255,107,0}] (-18.0,-3.00) -- (-18.0,-6.00);
\draw[->,rgb color={255,107,0}] (-18.0,-3.00) -- (7.00,-3.00) node[right]{};
\draw[->,rgb color={255,107,0}] (-18.0,-3.00) -- (7.00,-3.00) node[right]{};
\draw[->,rgb color={255,107,0}] (-18.0,-3.00) -- (7.00,-3.00) node[right]{};
\draw[->,rgb color={255,107,0}] (-18.0,-3.00) -- (7.00,-3.00) node[right]{};
\draw[->,rgb color={255,107,0}] (-18.0,-3.00) -- (7.00,-3.00) node[right]{};
\draw[->,rgb color={255,107,0}] (-18.0,-3.00) -- (7.00,-3.00) node[right]{};
\draw[->,rgb color={255,107,0}] (-15.0,-3.00) -- (7.00,-3.00) node[right]{};
\draw[->,rgb color={255,107,0}] (-15.0,-3.00) -- (7.00,-3.00) node[right]{};
\draw[->,rgb color={255,107,0}] (-18.0,-3.00) -- (7.00,-3.00) node[right]{};
\draw[->,rgb color={255,107,0}] (-18.0,-3.00) -- (7.00,-1.33);
\draw[->,rgb color={255,107,0}] (-18.0,-3.00) -- (7.00,-1.33);
\draw[->,rgb color={255,107,0}] (-18.0,-3.00) -- (7.00,-1.33);
\draw[->,rgb color={255,107,0}] (-18.0,-3.00) -- (7.00,-1.61);
\draw[->,rgb color={255,107,0}] (-18.0,-3.00) -- (7.00,-1.81);
\draw[->,rgb color={255,107,0}] (-18.0,-3.00) -- (7.00,-1.15);
\draw[->,rgb color={255,107,0}] (-18.0,4.00) -- (-40.5,7.00);
\draw[->,rgb color={255,107,0}] (-18.0,4.00) -- (-36.0,7.00);
\draw[->,rgb color={255,107,0}] (-18.0,4.00) -- (-36.0,7.00);
\draw[->,rgb color={255,107,0}] (-18.0,4.00) -- (-36.0,7.00);
\draw[->,rgb color={255,107,0}] (-18.0,4.00) -- (-27.0,7.00);
\draw[->,rgb color={255,107,0}] (-18.0,4.00) -- (-18.0,7.00);
\draw[->,rgb color={255,107,0}] (-18.0,4.00) -- (7.00,4.00) node[right]{};
\draw[->,rgb color={255,107,0}] (-18.0,4.00) -- (7.00,4.00) node[right]{};
\draw[->,rgb color={255,107,0}] (-18.0,4.00) -- (7.00,4.00) node[right]{};
\draw[->,rgb color={255,107,0}] (-18.0,4.00) -- (7.00,4.00) node[right]{};
\draw[->,rgb color={255,107,0}] (-18.0,4.00) -- (7.00,4.00) node[right]{};
\draw[->,rgb color={255,107,0}] (-18.0,4.00) -- (7.00,4.00) node[right]{};
\draw[->,rgb color={255,107,0}] (-15.0,4.00) -- (7.00,4.00) node[right]{};
\draw[->,rgb color={255,107,0}] (-15.0,4.00) -- (7.00,4.00) node[right]{};
\draw[->,rgb color={255,107,0}] (-18.0,4.00) -- (7.00,4.00) node[right]{};
\draw[->,rgb color={255,107,0}] (-18.0,4.00) -- (7.00,2.33);
\draw[->,rgb color={255,107,0}] (-18.0,4.00) -- (7.00,2.33);
\draw[->,rgb color={255,107,0}] (-18.0,4.00) -- (7.00,2.33);
\draw[->,rgb color={255,107,0}] (-18.0,4.00) -- (7.00,2.61);
\draw[->,rgb color={255,107,0}] (-18.0,4.00) -- (7.00,2.81);
\draw[->,rgb color={255,107,0}] (-18.0,4.00) -- (7.00,2.15);
\draw[->,rgb color={255,107,0}] (-15.0,-3.00) -- (-24.0,-6.00);
\draw[->,rgb color={255,107,0}] (-15.0,-3.00) -- (7.00,-1.78);
\draw[->,rgb color={255,107,0}] (-15.0,4.00) -- (-24.0,7.00);
\draw[->,rgb color={255,107,0}] (-15.0,4.00) -- (7.00,2.78);
\draw[->,rgb color={255,107,0}] (-11.2,-2.75) -- (7.00,-2.75) node[right]{};
\draw[->,rgb color={255,107,0}] (-13.5,-2.75) -- (7.00,-2.75) node[right]{};
\draw[->,rgb color={255,0,0}] (-13.5,-2.50) -- (-45.0,-6.00);
\draw[->,rgb color={255,0,0}] (-13.5,-2.50) -- (7.00,-0.222);
\draw[->,rgb color={255,0,0}] (-13.5,3.50) -- (-45.0,7.00);
\draw[->,rgb color={255,0,0}] (-13.5,3.50) -- (7.00,1.22);
\draw[->,rgb color={255,77,0}] (-11.2,3.75) -- (7.00,3.75) node[right]{};
\draw[->,rgb color={255,77,0}] (-13.5,3.75) -- (7.00,3.75) node[right]{};
\draw[->,rgb color={255,107,0}] (-13.0,-2.67) -- (7.00,-2.67) node[right]{};
\draw[->,rgb color={255,107,0}] (-13.0,-2.67) -- (7.00,-2.67) node[right]{};
\draw[->,rgb color={255,77,0}] (-13.0,3.67) -- (7.00,3.67) node[right]{};
\draw[->,rgb color={255,107,0}] (-12.0,-2.50) -- (-12.0,-6.00);
\draw[->,rgb color={255,107,0}] (-12.0,-2.50) -- (7.00,-2.50) node[right]{};
\draw[->,rgb color={255,107,0}] (-12.0,-2.50) -- (7.00,-2.50) node[right]{};
\draw[->,rgb color={255,107,0}] (-10.5,-2.50) -- (7.00,-2.50) node[right]{};
\draw[->,rgb color={255,107,0}] (-7.50,-2.50) -- (7.00,-2.50) node[right]{};
\draw[->,rgb color={255,107,0}] (-12.0,-2.50) -- (7.00,-2.50) node[right]{};
\draw[->,rgb color={255,107,0}] (-12.0,-2.50) -- (7.00,-1.44);
\draw[->,rgb color={255,77,0}] (-12.0,3.50) -- (-12.0,7.00);
\draw[->,rgb color={255,77,0}] (-12.0,3.50) -- (7.00,3.50) node[right]{};
\draw[->,rgb color={255,77,0}] (-12.0,3.50) -- (7.00,3.50) node[right]{};
\draw[->,rgb color={255,77,0}] (-10.5,3.50) -- (7.00,3.50) node[right]{};
\draw[->,rgb color={255,77,0}] (-7.50,3.50) -- (7.00,3.50) node[right]{};
\draw[->,rgb color={255,77,0}] (-12.0,3.50) -- (7.00,3.50) node[right]{};
\draw[->,rgb color={255,77,0}] (-12.0,3.50) -- (7.00,2.44);
\draw[->,rgb color={255,107,0}] (-10.0,-2.33) -- (7.00,-2.33) node[right]{};
\draw[->,rgb color={255,77,0}] (-10.0,3.33) -- (7.00,3.33) node[right]{};
\draw[->,rgb color={255,77,0}] (-10.0,3.33) -- (7.00,3.33) node[right]{};
\draw[->,rgb color={255,107,0}] (-6.75,-2.25) -- (7.00,-2.25) node[right]{};
\draw[->,rgb color={255,107,0}] (-9.00,-2.25) -- (7.00,-2.25) node[right]{};
\draw[->,rgb color={255,77,0}] (-9.00,-2.00) -- (-27.0,-6.00);
\draw[->,rgb color={255,77,0}] (-9.00,-2.00) -- (-21.0,-6.00);
\draw[->,rgb color={255,77,0}] (-9.00,-2.00) -- (-21.0,-6.00);
\draw[->,rgb color={255,77,0}] (-9.00,-2.00) -- (-21.0,-6.00);
\draw[->,rgb color={255,77,0}] (-9.00,-2.00) -- (-9.00,-6.00);
\draw[->,rgb color={255,77,0}] (-9.00,-2.00) -- (7.00,-2.00) node[right]{};
\draw[->,rgb color={255,77,0}] (-9.00,-2.00) -- (7.00,-2.00) node[right]{};
\draw[->,rgb color={255,77,0}] (-9.00,-2.00) -- (7.00,-2.00) node[right]{};
\draw[->,rgb color={255,77,0}] (-9.00,-2.00) -- (7.00,-2.00) node[right]{};
\draw[->,rgb color={255,77,0}] (-9.00,-2.00) -- (7.00,-2.00) node[right]{};
\draw[->,rgb color={255,77,0}] (-9.00,-2.00) -- (7.00,-2.00) node[right]{};
\draw[->,rgb color={255,77,0}] (-6.00,-2.00) -- (7.00,-2.00) node[right]{};
\draw[->,rgb color={255,77,0}] (-6.00,-2.00) -- (7.00,-2.00) node[right]{};
\draw[->,rgb color={255,77,0}] (-9.00,-2.00) -- (7.00,-2.00) node[right]{};
\draw[->,rgb color={255,77,0}] (-9.00,-2.00) -- (3.00,-6.00);
\draw[->,rgb color={255,77,0}] (-9.00,-2.00) -- (7.00,-0.667);
\draw[->,rgb color={255,77,0}] (-9.00,-2.00) -- (7.00,-0.667);
\draw[->,rgb color={255,77,0}] (-9.00,-2.00) -- (7.00,-0.667);
\draw[->,rgb color={255,77,0}] (-9.00,-2.00) -- (7.00,-0.933);
\draw[->,rgb color={255,77,0}] (-9.00,-2.00) -- (7.00,-1.11);
\draw[->,rgb color={255,77,0}] (-9.00,-2.00) -- (7.00,-0.476);
\draw[->,rgb color={255,77,0}] (-9.00,3.00) -- (-27.0,7.00);
\draw[->,rgb color={255,77,0}] (-9.00,3.00) -- (-21.0,7.00);
\draw[->,rgb color={255,77,0}] (-9.00,3.00) -- (-21.0,7.00);
\draw[->,rgb color={255,77,0}] (-9.00,3.00) -- (-21.0,7.00);
\draw[->,rgb color={255,77,0}] (-9.00,3.00) -- (-9.00,7.00);
\draw[->,rgb color={255,77,0}] (-9.00,3.00) -- (7.00,3.00) node[right]{};
\draw[->,rgb color={255,77,0}] (-9.00,3.00) -- (7.00,3.00) node[right]{};
\draw[->,rgb color={255,77,0}] (-9.00,3.00) -- (7.00,3.00) node[right]{};
\draw[->,rgb color={255,77,0}] (-9.00,3.00) -- (7.00,3.00) node[right]{};
\draw[->,rgb color={255,77,0}] (-9.00,3.00) -- (7.00,3.00) node[right]{};
\draw[->,rgb color={255,77,0}] (-9.00,3.00) -- (7.00,3.00) node[right]{};
\draw[->,rgb color={255,77,0}] (-6.00,3.00) -- (7.00,3.00) node[right]{};
\draw[->,rgb color={255,77,0}] (-6.00,3.00) -- (7.00,3.00) node[right]{};
\draw[->,rgb color={255,77,0}] (-9.00,3.00) -- (7.00,3.00) node[right]{};
\draw[->,rgb color={255,77,0}] (-9.00,3.00) -- (3.00,7.00);
\draw[->,rgb color={255,77,0}] (-9.00,3.00) -- (7.00,1.67);
\draw[->,rgb color={255,77,0}] (-9.00,3.00) -- (7.00,1.67);
\draw[->,rgb color={255,77,0}] (-9.00,3.00) -- (7.00,1.67);
\draw[->,rgb color={255,77,0}] (-9.00,3.00) -- (7.00,1.93);
\draw[->,rgb color={255,77,0}] (-9.00,3.00) -- (7.00,2.11);
\draw[->,rgb color={255,77,0}] (-9.00,3.00) -- (7.00,1.48);
\draw[->,rgb color={255,77,0}] (-6.75,3.25) -- (7.00,3.25) node[right]{};
\draw[->,rgb color={255,77,0}] (-9.00,3.25) -- (7.00,3.25) node[right]{};
\draw[->,rgb color={255,77,0}] (-6.00,-2.00) -- (-6.00,-6.00);
\draw[->,rgb color={255,77,0}] (-6.00,-2.00) -- (7.00,-1.13);
\draw[->,rgb color={255,0,0}] (-6.00,-1.50) -- (-33.0,-6.00);
\draw[->,rgb color={255,0,0}] (-6.00,-1.50) -- (7.00,0.667);
\draw[->,rgb color={255,0,0}] (-6.00,2.50) -- (-33.0,7.00);
\draw[->,rgb color={255,0,0}] (-6.00,2.50) -- (7.00,0.333);
\draw[->,rgb color={255,77,0}] (-6.00,3.00) -- (-6.00,7.00);
\draw[->,rgb color={255,77,0}] (-6.00,3.00) -- (7.00,2.13);
\draw[->,rgb color={255,77,0}] (-3.00,-1.75) -- (7.00,-1.75) node[right]{};
\draw[->,rgb color={255,77,0}] (-5.25,-1.75) -- (7.00,-1.75) node[right]{};
\draw[->,rgb color={255,42,0}] (-3.00,2.75) -- (7.00,2.75) node[right]{};
\draw[->,rgb color={255,42,0}] (-5.25,2.75) -- (7.00,2.75) node[right]{};
\draw[->,rgb color={255,77,0}] (-5.00,-1.67) -- (7.00,-1.67) node[right]{};
\draw[->,rgb color={255,42,0}] (-5.00,2.67) -- (7.00,2.67) node[right]{};
\draw[->,rgb color={255,77,0}] (-4.50,-1.50) -- (7.00,-1.50) node[right]{};
\draw[->,rgb color={255,77,0}] (-4.50,-1.50) -- (7.00,-1.50) node[right]{};
\draw[->,rgb color={255,77,0}] (-3.00,-1.50) -- (7.00,-1.50) node[right]{};
\draw[->,rgb color={255,77,0}] (0.000,-1.50) -- (7.00,-1.50) node[right]{};
\draw[->,rgb color={255,77,0}] (-4.50,-1.50) -- (7.00,-1.50) node[right]{};
\draw[->,rgb color={255,77,0}] (-4.50,-1.50) -- (7.00,-5.33);
\draw[->,rgb color={255,77,0}] (-4.50,-1.50) -- (7.00,-0.733);
\draw[->,rgb color={255,42,0}] (-4.50,2.50) -- (7.00,2.50) node[right]{};
\draw[->,rgb color={255,42,0}] (-4.50,2.50) -- (7.00,2.50) node[right]{};
\draw[->,rgb color={255,42,0}] (-3.00,2.50) -- (7.00,2.50) node[right]{};
\draw[->,rgb color={255,42,0}] (0.000,2.50) -- (7.00,2.50) node[right]{};
\draw[->,rgb color={255,42,0}] (-4.50,2.50) -- (7.00,2.50) node[right]{};
\draw[->,rgb color={255,42,0}] (-4.50,2.50) -- (7.00,6.33);
\draw[->,rgb color={255,42,0}] (-4.50,2.50) -- (7.00,1.73);
\draw[->,rgb color={255,77,0}] (-3.00,-1.33) -- (7.00,-1.33) node[right]{};
\draw[->,rgb color={255,42,0}] (-3.00,-1.00) -- (-10.5,-6.00);
\draw[->,rgb color={255,42,0}] (-3.00,-1.00) -- (-3.00,-6.00);
\draw[->,rgb color={255,42,0}] (-3.00,-1.00) -- (-3.00,-6.00);
\draw[->,rgb color={255,42,0}] (-3.00,-1.00) -- (-3.00,-6.00);
\draw[->,rgb color={255,42,0}] (-3.00,-1.00) -- (7.00,-1.00) node[right]{};
\draw[->,rgb color={255,42,0}] (-3.00,-1.00) -- (7.00,-1.00) node[right]{};
\draw[->,rgb color={255,42,0}] (-3.00,-1.00) -- (7.00,-1.00) node[right]{};
\draw[->,rgb color={255,42,0}] (-3.00,-1.00) -- (7.00,-1.00) node[right]{};
\draw[->,rgb color={255,42,0}] (-3.00,-1.00) -- (7.00,-1.00) node[right]{};
\draw[->,rgb color={255,42,0}] (-3.00,-1.00) -- (7.00,-1.00) node[right]{};
\draw[->,rgb color={255,42,0}] (0.000,-1.00) -- (7.00,-1.00) node[right]{};
\draw[->,rgb color={255,42,0}] (0.000,-1.00) -- (7.00,-1.00) node[right]{};
\draw[->,rgb color={255,42,0}] (-3.00,-1.00) -- (7.00,-1.00) node[right]{};
\draw[->,rgb color={255,42,0}] (-3.00,-1.00) -- (7.00,-4.33);
\draw[->,rgb color={255,42,0}] (-3.00,-1.00) -- (7.00,-2.67);
\draw[->,rgb color={255,42,0}] (-3.00,-1.00) -- (7.00,0.111);
\draw[->,rgb color={255,42,0}] (-3.00,-1.00) -- (7.00,0.111);
\draw[->,rgb color={255,42,0}] (-3.00,-1.00) -- (7.00,0.111);
\draw[->,rgb color={255,42,0}] (-3.00,-1.00) -- (7.00,-0.167);
\draw[->,rgb color={255,42,0}] (-3.00,-1.00) -- (7.00,-0.333);
\draw[->,rgb color={255,42,0}] (-3.00,-1.00) -- (7.00,0.333);
\draw[->,rgb color={255,42,0}] (-3.00,2.00) -- (-10.5,7.00);
\draw[->,rgb color={255,42,0}] (-3.00,2.00) -- (-3.00,7.00);
\draw[->,rgb color={255,42,0}] (-3.00,2.00) -- (-3.00,7.00);
\draw[->,rgb color={255,42,0}] (-3.00,2.00) -- (-3.00,7.00);
\draw[->,rgb color={255,42,0}] (-3.00,2.00) -- (7.00,2.00) node[right]{};
\draw[->,rgb color={255,42,0}] (-3.00,2.00) -- (7.00,2.00) node[right]{};
\draw[->,rgb color={255,42,0}] (-3.00,2.00) -- (7.00,2.00) node[right]{};
\draw[->,rgb color={255,42,0}] (-3.00,2.00) -- (7.00,2.00) node[right]{};
\draw[->,rgb color={255,42,0}] (-3.00,2.00) -- (7.00,2.00) node[right]{};
\draw[->,rgb color={255,42,0}] (-3.00,2.00) -- (7.00,2.00) node[right]{};
\draw[->,rgb color={255,42,0}] (0.000,2.00) -- (7.00,2.00) node[right]{};
\draw[->,rgb color={255,42,0}] (0.000,2.00) -- (7.00,2.00) node[right]{};
\draw[->,rgb color={255,42,0}] (-3.00,2.00) -- (7.00,2.00) node[right]{};
\draw[->,rgb color={255,42,0}] (-3.00,2.00) -- (7.00,5.33);
\draw[->,rgb color={255,42,0}] (-3.00,2.00) -- (7.00,3.67);
\draw[->,rgb color={255,42,0}] (-3.00,2.00) -- (7.00,0.889);
\draw[->,rgb color={255,42,0}] (-3.00,2.00) -- (7.00,0.889);
\draw[->,rgb color={255,42,0}] (-3.00,2.00) -- (7.00,0.889);
\draw[->,rgb color={255,42,0}] (-3.00,2.00) -- (7.00,1.17);
\draw[->,rgb color={255,42,0}] (-3.00,2.00) -- (7.00,0.667);
\draw[->,rgb color={255,42,0}] (-3.00,2.00) -- (7.00,1.33);
\draw[->,rgb color={255,42,0}] (-3.00,2.33) -- (7.00,2.33) node[right]{};
\draw[->,rgb color={255,42,0}] (-3.00,2.33) -- (7.00,2.33) node[right]{};
\draw[->,rgb color={255,77,0}] (0.000,-1.25) -- (7.00,-1.25) node[right]{};
\draw[->,rgb color={255,77,0}] (-2.25,-1.25) -- (7.00,-1.25) node[right]{};
\draw[->,rgb color={255,42,0}] (0.000,2.25) -- (7.00,2.25) node[right]{};
\draw[->,rgb color={255,42,0}] (-2.25,2.25) -- (7.00,2.25) node[right]{};
\draw[->,rgb color={255,0,0}] (-1.50,-0.500) -- (-18.0,-6.00);
\draw[->,rgb color={255,0,0}] (-1.50,-0.500) -- (7.00,2.33);
\draw[->,rgb color={255,0,0}] (-1.50,1.50) -- (-18.0,7.00);
\draw[->,rgb color={255,0,0}] (-1.50,1.50) -- (7.00,-1.33);
\draw[->,rgb color={255,42,0}] (0.000,-1.00) -- (7.00,-3.33);
\draw[->,rgb color={255,42,0}] (0.000,-1.00) -- (7.00,-0.417);
\draw[->,rgb color={255,42,0}] (2.25,-0.750) -- (7.00,-0.750) node[right]{};
\draw[->,rgb color={255,42,0}] (0.000,-0.750) -- (7.00,-0.750) node[right]{};
\draw[->,rgb color={255,42,0}] (0.000,-0.667) -- (7.00,-0.667) node[right]{};
\draw[->,rgb color={255,42,0}] (0.000,-0.500) -- (7.00,-0.500) node[right]{};
\draw[->,rgb color={255,42,0}] (0.000,-0.500) -- (7.00,-0.500) node[right]{};
\draw[->,rgb color={255,42,0}] (1.50,-0.500) -- (7.00,-0.500) node[right]{};
\draw[->,rgb color={255,42,0}] (4.50,-0.500) -- (7.00,-0.500) node[right]{};
\draw[->,rgb color={255,42,0}] (0.000,-0.500) -- (7.00,-0.500) node[right]{};
\draw[->,rgb color={255,42,0}] (0.000,-0.500) -- (7.00,-1.67);
\draw[->,rgb color={255,42,0}] (0.000,-0.500) -- (7.00,0.0833);
\draw[->,rgb color={255,0,0}] (0.000,0.000) -- (7.00,0.000) node[right]{};
\draw[->,rgb color={255,0,0}] (0.000,0.000) -- (7.00,0.000) node[right]{};
\draw[->,rgb color={255,0,0}] (0.000,0.000) -- (7.00,0.000) node[right]{};
\draw[->,rgb color={255,0,0}] (0.000,0.000) -- (7.00,0.000) node[right]{};
\draw[->,rgb color={255,0,0}] (0.000,0.000) -- (7.00,0.000) node[right]{};
\draw[->,rgb color={255,0,0}] (0.000,0.000) -- (7.00,0.000) node[right]{};
\draw[->,rgb color={255,0,0}] (3.00,0.000) -- (7.00,0.000) node[right]{};
\draw[->,rgb color={255,0,0}] (3.00,0.000) -- (7.00,0.000) node[right]{};
\draw[->,rgb color={255,0,0}] (0.000,0.000) -- (7.00,0.000) node[right]{};
\draw[->,rgb color={255,0,0}] (0.000,0.000) -- (7.00,-4.67);
\draw[->,rgb color={255,0,0}] (0.000,0.000) -- (7.00,-2.33);
\draw[->,rgb color={255,0,0}] (0.000,0.000) -- (7.00,-2.33);
\draw[->,rgb color={255,0,0}] (0.000,0.000) -- (7.00,-2.33);
\draw[->,rgb color={255,0,0}] (0.000,0.000) -- (7.00,-1.17);
\draw[->,rgb color={255,0,0}] (0.000,0.000) -- (7.00,1.17);
\draw[->,rgb color={255,0,0}] (0.000,0.000) -- (7.00,1.17);
\draw[->,rgb color={255,0,0}] (0.000,0.000) -- (7.00,1.17);
\draw[->,rgb color={255,0,0}] (0.000,0.000) -- (7.00,-0.778);
\draw[->,rgb color={255,0,0}] (0.000,0.000) -- (7.00,0.778);
\draw[->,rgb color={255,0,0}] (0.000,0.000) -- (7.00,1.56);
\draw[->,rgb color={255,0,0}] (0.000,0.000) -- (7.00,0.583);
\draw[->,rgb color={255,0,0}] (0.000,0.500) -- (0.000,-6.00);
\draw[->,rgb color={255,0,0}] (0.000,0.500) -- (0.000,7.00);
\draw[->,rgb color={255,0,0}] (0.000,1.00) -- (7.00,1.00) node[right]{};
\draw[->,rgb color={255,0,0}] (0.000,1.00) -- (7.00,1.00) node[right]{};
\draw[->,rgb color={255,0,0}] (0.000,1.00) -- (7.00,1.00) node[right]{};
\draw[->,rgb color={255,0,0}] (0.000,1.00) -- (7.00,1.00) node[right]{};
\draw[->,rgb color={255,0,0}] (0.000,1.00) -- (7.00,1.00) node[right]{};
\draw[->,rgb color={255,0,0}] (0.000,1.00) -- (7.00,1.00) node[right]{};
\draw[->,rgb color={255,0,0}] (3.00,1.00) -- (7.00,1.00) node[right]{};
\draw[->,rgb color={255,0,0}] (3.00,1.00) -- (7.00,1.00) node[right]{};
\draw[->,rgb color={255,0,0}] (0.000,1.00) -- (7.00,1.00) node[right]{};
\draw[->,rgb color={255,0,0}] (0.000,1.00) -- (7.00,3.33);
\draw[->,rgb color={255,0,0}] (0.000,1.00) -- (7.00,3.33);
\draw[->,rgb color={255,0,0}] (0.000,1.00) -- (7.00,3.33);
\draw[->,rgb color={255,0,0}] (0.000,1.00) -- (7.00,5.67);
\draw[->,rgb color={255,0,0}] (0.000,1.00) -- (7.00,-0.167);
\draw[->,rgb color={255,0,0}] (0.000,1.00) -- (7.00,-0.167);
\draw[->,rgb color={255,0,0}] (0.000,1.00) -- (7.00,-0.167);
\draw[->,rgb color={255,0,0}] (0.000,1.00) -- (7.00,2.17);
\draw[->,rgb color={255,0,0}] (0.000,1.00) -- (7.00,-0.556);
\draw[->,rgb color={255,0,0}] (0.000,1.00) -- (7.00,0.222);
\draw[->,rgb color={255,0,0}] (0.000,1.00) -- (7.00,1.78);
\draw[->,rgb color={255,0,0}] (0.000,1.00) -- (7.00,0.417);
\draw[->,rgb color={255,0,0}] (0.000,1.50) -- (7.00,1.50) node[right]{};
\draw[->,rgb color={255,0,0}] (0.000,1.50) -- (7.00,1.50) node[right]{};
\draw[->,rgb color={255,0,0}] (1.50,1.50) -- (7.00,1.50) node[right]{};
\draw[->,rgb color={255,0,0}] (4.50,1.50) -- (7.00,1.50) node[right]{};
\draw[->,rgb color={255,0,0}] (0.000,1.50) -- (7.00,1.50) node[right]{};
\draw[->,rgb color={255,0,0}] (0.000,1.50) -- (7.00,2.67);
\draw[->,rgb color={255,0,0}] (0.000,1.50) -- (7.00,0.917);
\draw[->,rgb color={255,0,0}] (0.000,1.67) -- (7.00,1.67) node[right]{};
\draw[->,rgb color={255,0,0}] (2.25,1.75) -- (7.00,1.75) node[right]{};
\draw[->,rgb color={255,0,0}] (0.000,1.75) -- (7.00,1.75) node[right]{};
\draw[->,rgb color={255,42,0}] (0.000,2.00) -- (7.00,4.33);
\draw[->,rgb color={255,42,0}] (0.000,2.00) -- (7.00,1.42);
\draw[->,rgb color={255,42,0}] (1.00,-0.333) -- (7.00,-0.333) node[right]{};
\draw[->,rgb color={255,0,0}] (1.00,1.33) -- (7.00,1.33) node[right]{};
\draw[->,rgb color={255,0,0}] (1.00,1.33) -- (7.00,1.33) node[right]{};
\draw[->,rgb color={255,42,0}] (3.75,-0.250) -- (7.00,-0.250) node[right]{};
\draw[->,rgb color={255,42,0}] (1.50,-0.250) -- (7.00,-0.250) node[right]{};
\draw[->,rgb color={255,0,0}] (1.50,0.500) -- (7.00,0.500) node[right]{};
\draw[->,rgb color={255,0,0}] (1.50,0.500) -- (7.00,0.500) node[right]{};
\draw[->,rgb color={255,0,0}] (3.00,0.500) -- (7.00,0.500) node[right]{};
\draw[->,rgb color={255,0,0}] (6.00,0.500) -- (7.00,0.500) node[right]{};
\draw[->,rgb color={255,0,0}] (1.50,0.500) -- (7.00,0.500) node[right]{};
\draw[->,rgb color={255,0,0}] (1.50,0.500) -- (7.00,-0.111);
\draw[->,rgb color={255,0,0}] (1.50,0.500) -- (7.00,1.11);
\draw[->,rgb color={255,0,0}] (3.75,1.25) -- (7.00,1.25) node[right]{};
\draw[->,rgb color={255,0,0}] (1.50,1.25) -- (7.00,1.25) node[right]{};
\draw[->,rgb color={255,0,0}] (2.00,0.333) -- (7.00,0.333) node[right]{};
\draw[->,rgb color={255,0,0}] (2.00,0.333) -- (7.00,0.333) node[right]{};
\draw[->,rgb color={255,0,0}] (2.00,0.667) -- (7.00,0.667) node[right]{};
\draw[->,rgb color={255,0,0}] (4.50,0.250) -- (7.00,0.250) node[right]{};
\draw[->,rgb color={255,0,0}] (2.25,0.250) -- (7.00,0.250) node[right]{};
\draw[->,rgb color={255,0,0}] (4.50,0.750) -- (7.00,0.750) node[right]{};
\draw[->,rgb color={255,0,0}] (2.25,0.750) -- (7.00,0.750) node[right]{};
\draw[->,rgb color={255,0,0}] (3.00,0.000) -- (7.00,-0.667);
\draw[->,rgb color={255,0,0}] (3.00,0.000) -- (7.00,0.444);
\draw[->,rgb color={255,0,0}] (3.00,1.00) -- (7.00,1.67);
\draw[->,rgb color={255,0,0}] (3.00,1.00) -- (7.00,0.556);
\end{tikzpicture}
}
}
\caption{First steps of the scattering diagram $\tilde{S}(\fD_{\PP^2/E})$. Figure due to Tim Gabele \cite{gabele2019tropical}.}
\label{figure_scattering}
\end{figure}

\end{document}